\documentclass[12pt]{amsart}
\usepackage{amsfonts,latexsym,rawfonts,amsmath,amssymb,amsthm,mathrsfs}
\usepackage[colorlinks=true,urlcolor=blue,linkcolor=blue,anchorcolor=blue,citecolor=blue]{hyperref}
\usepackage{graphicx}
\usepackage{comment}
\usepackage[]{caption2}

\usepackage[pagewise]{lineno}
\usepackage{amsthm}
\usepackage{lipsum}
\usepackage[raggedright]{sidecap}\sidecaptionvpos{figure}{t} 
\usepackage{tikz}
\usepackage{amssymb}
\usepackage{amsmath}
\usepackage{enumerate}
\newtheorem{mainthm}{Theorem}

\allowdisplaybreaks[4]

\makeatletter
\@namedef{subjclassname@2020}{\textup{2020} Mathematics Subject Classification}
\makeatother

\numberwithin{equation}{section}

\RequirePackage{color}
\theoremstyle{plain}

\newtheorem{theorem}{Theorem}[section]
\newtheorem{lemma}{Lemma}[section]
\newtheorem{corollary}{Corollary}[section]

\newtheorem{remark}{Remark}[section]

\begin{document}
\title{On the new weighted geometric inequalities near the sphere in space forms}

\author{Weimin Sheng}
\address{Weimin Sheng: School of Mathematical Sciences, Zhejiang University, Hangzhou 310027, China.}
\email{shengweimin@zju.edu.cn}

\author{Yinhang Wang}
\address{Yinhang Wang: School of Mathematical Sciences, Zhejiang University, Hangzhou 310027, China.}
\email{wangyinhang@zju.edu.cn}


\begin{abstract}
In this paper, we first investigate weighted Minkowski type inequalities for nearly spherical sets in space forms, focusing on the sets that are $C^1$-close to geodesic spheres. Our results generalize the work of \cite{G22} by incorporating broader geometric settings and convex weight functions. Additionally, we establish quantitative stability estimates for weighted Alexandrov-Fenchel type inequalities in $\mathbb{R}^{n+1}$ and $\mathbb{H}^{n+1}$, extending the earlier results of \cite{VW24} and \cite{ZZ23}. These inequalities hold for nearly spherical sets that are $W^{2,\infty}$-close to geodesic spheres coupled with general convex weights.
\end{abstract}

\maketitle

\baselineskip16pt
\parskip3pt

\section{Introduction}
The classical Alexandrov-Fenchel inequalities state that for a convex body $\Omega\subset\mathbb{R}^{n+1}$ with smooth boundary $M=\partial\Omega$,
\begin{equation}\label{AF ineq}
	\int_{M}\sigma_{k}d\mu\geq\binom{n}{k}\omega_n\left(\frac{\int_{M}\sigma_{\ell}d\mu}{\binom{n}{\ell}\omega_n}\right)^{\frac{n-k}{n-\ell}}, \quad 0\leq\ell<k\leq n,
\end{equation}
where $\sigma_k$ is the $k$th mean curvature of $M$ and $\omega_n$ is the area of the unit sphere $\mathbb{S}^n$. Equality holds in (\ref{AF ineq}) if and only if $M$ is a sphere. The Alexandrov-Fenchel inequalities (\ref{AF ineq}) have been generalized to non-convex domains, including the star-shaped and $k$-convex domains, through the work of Guan-Li \cite{GL09} by using a inverse curvature flow. Further extensions and related results can be found in \cite{CW13,CW14,Q15}. 

Recently, the study of weighted geometric inequalities has attracted a lot of attentions.
Kwong-Miao \cite{KM14} first obtained the inequality
\begin{equation}\label{KM ineq}
	\int_{M}\Phi Hd\mu \geq \frac{n(n+1)}{2}\operatorname{Vol}(\Omega),
\end{equation}
for star-shaped and mean convex hypersurface $M\subset\mathbb{R}^{n+1}$ enclosing a bounded domain $\Omega$, where $\Phi=\frac{1}{2}r^2$. Equality holds in (\ref{KM ineq}) if and only if $M$ is a round sphere centered at the origin. Later, the inequality (\ref{KM ineq}) has been generalized to higher order mean curvatures in \cite{GR20,KM15,WZ23}. Wu \cite{W24} extended it to a broader weighted class $\Phi^{\alpha}$ $(\alpha\geq1)$ via normalized inverse curvature flow. Kwong-Wei \cite{KW24} further generalized it to arbitrary non-decreasing convex $C^2$ positive weight function $g(\Phi)$, and they proved for any smooth closed star-shaped and $k$-convex hypersurface $M\subset\mathbb{R}^{n+1}$,
\begin{equation}\label{KW gphi ineq}
	\begin{aligned}
		\int_M g(\Phi)\sigma_{k}d\mu \geq \binom{n}{k}\omega_n g\left(\frac{1}{2}\left(\frac{\int_{M}\sigma_{\ell}d\mu}{\binom{n}{\ell}\omega_n}\right)^{\frac{2}{n-\ell}}\right)\left(\frac{\int_{M}\sigma_{\ell}d\mu}{\binom{n}{\ell}\omega_n}\right)^{\frac{n-k}{n-\ell}}, \quad 0\leq \ell < k \leq n.
	\end{aligned}
\end{equation}
When $g(\cdot)\equiv1$, inequalities (\ref{KW gphi ineq}) reduce to the classical Alexandrov-Fenchel inequalities (\ref{AF ineq}).

To state the weighted Alexandrov-Fenchel type inequalities in general space forms $N^{n+1}(K)$ with constant sectional curvature $K=-1,0$ or $1$, we consider the space forms as the warped product manifolds $N^{n+1}(K)=I \times \mathbb{S}^n$ equipped with the metric
$$
\bar{g}=d r^2+\phi^2(r) g_{\mathbb{S}^n},
$$
where $g_{\mathbb{S}^n}$ is the standard round metric of $\mathbb{S}^n$. More specifically, $N^{n+1}(0)$ is the Euclidean space $\mathbb{R}^{n+1}$ if $\phi(r)=r$, $r\in[0,\infty)$; $N^{n+1}(-1)$ is the hyperbolic space $\mathbb{H}^{n+1}$ if $\phi(r)=\sinh r$, $r\in[0,\infty)$; and $N^{n+1}(1)$ is the sphere $\mathbb{S}^{n+1}$ if $\phi(r)=\sin r$, $r\in[0,\pi)$. We define
\begin{equation}
	\Phi(r)=\int_{0}^{r}\phi(s)ds=\left\{\begin{array}{lll}
		\cosh r-1, & K=-1, \\
		\frac{1}{2} r^2, & K=0, \\
		1-\cos r, & K=1.
	\end{array}\right.
\end{equation}
It is well-known that the vector field $V=\bar{\nabla}\Phi=\phi(r)\partial_r$ on $N^{n+1}(K)$ is a conformal Killing field, satisfying $\bar{\nabla}(\phi(r)\partial_r)=\phi'(r)\bar{g}$.

For a smooth convex body $\Omega\subset N^{n+1}(K)$ with boundary $M=\partial\Omega$, the $k$th quermassintegral $W_{k}$ of $\Omega$ can be expressed as follows (see \cite{S06}):
$$
\begin{aligned}
	& W_{-1}(\Omega)=\operatorname{Vol}(\Omega), \quad W_0(\Omega)=|\partial\Omega|, \\
	& W_1(\Omega)=\int_{M} \sigma_1 d \mu + Kn\operatorname{Vol}(\Omega), \\
	& W_k(\Omega)=\int_{M} \sigma_k d \mu + K\frac{n-k+1}{k-1}W_{k-2}(\Omega), \quad k=2, \cdots, n,
\end{aligned}
$$
where $\sigma_k=\sigma_k(\kappa)$ is the $k$th elementary symmetric polynomial of the principal curvatures $\kappa=\left(\kappa_1, \cdots, \kappa_{n}\right)$ of $M$. 

In hyperbolic space $\mathbb{H}^{n+1}$, due to the close relationship with the quasilocal mass and the Riemannian Penrose inequality (see for instance \cite{BHW16}), the weighted Alexandrov-Fenchel type inequalities have been studied extensively. For h-convex hypersurfaces (i.e. $\kappa_i \geq 1$ for all principal curvatures $\kappa_i$), Hu-Li-Wei \cite{HLW22} established the following inequalities comparing weighted curvature integrals to quermassintegrals,
\begin{equation}\label{HLW ineq}
	\int_M \phi^{\prime}(r) \sigma_k d\mu \geq \eta_{k, \ell}\left(W_\ell(\Omega)\right), \quad -1\leq \ell <k\leq n,
\end{equation}
where $\eta_{k, \ell}$ is the unique function such that equality holds in (\ref{HLW ineq}) for geodesic spheres, and equality holds if and only if $M$ is a geodesic sphere centered at the origin. This result generalizes earlier work by Ge-Wang-Wu \cite{GWW15}, who proved the case for $k$ being odd and $\ell=0$. As an application, they proved a positive mass theorem for the Gauss-Bonnet-Chern mass of asymptotically hyperbolic graphs. For $k=1$ and $\ell=0$, the corresponding weighted Minkowski inequality was established in \cite{DG16} for star-shaped and mean convex hypersurfaces. See also \cite{BGL} for related results. Wei-Zhou \cite{WZ23} subsequently proved the following improved weighted geometric inequalities for h-convex hypersurfaces,
\begin{equation}\label{WZ ineq}
	\int_M \Phi \sigma_k d \mu \geq \xi_{k, \ell}\left(W_\ell(\Omega)\right), \quad -1\leq \ell < k \leq n,
\end{equation}
where $\xi_{k, \ell}$ is the unique function such that equality holds in (\ref{WZ ineq}) for geodesic spheres, and equality holds if and only if $M$ is a geodesic sphere centered at the origin. In this direction, we refer to \cite{KW23}.

In sphere $\mathbb{S}^{n+1}$, there are less results concerning weighted geometric inequalities. Applying Guan-Li's  mean curvature type flow \cite{GL15}, Wei-Zhou \cite{WZ23} proved for convex hypersurfaces,
\begin{equation}\label{sphere ineq}
	\int_{M}\Phi Hd\mu \geq \psi (\operatorname{Vol}(\Omega)),
\end{equation}
where $\psi$ is the unique function such that equality holds in (\ref{sphere ineq}) for geodesic spheres, and equality holds if and only if $M$ is a geodesic sphere centered at the origin. For other related results, see \cite{GP17,KWW22} and references therein.

Based on the aforementioned works, a natural question is whether there exists a family of Alexandrov-Fenchel type inequalities involving more general weight function in space forms $N^{n+1}(K)$. In this paper, we focus on the validity and the stability of various weighted geometric inequalities for nearly spherical sets in space forms $N^{n+1}(K)$.

Let $\Omega\subset N^{n+1}(K)$ be a star-shaped domain with respect to the origin $O$. Denote $M=\partial\Omega$, and let $u:\mathbb{S}^{n}\rightarrow(-1,+\infty)$ be a function such that $M$ can be parametrized as $M=\{(\rho(1+u(x)),x):x\in \mathbb{S}^n\}$, where $\rho>0$ is a constant. When $u=0$, it reduces to the geodesic ball of radius $\rho$ centered at $O$, denoted by $\bar{B}_{\rho}$. The barycenter of $\Omega$, denoted by $\operatorname{bar}(\Omega)$, is defined as the minimizer $p \in N^{n+1}(K)$ of the function (see for instance in \cite{BDS15,BDF17,ZZ23})
$$
p \mapsto \int_{\Omega} \mathrm{d}_K^2(y, p) d v_K(y),
$$
where $\mathrm{d}_K(y, p)$ is the geodesic distance between $y \in \Omega$ and $p$, and $d v_K(y)$ is the volume measure of $N^{n+1}(K)$.

Motivated by the analysis in \cite{F89}, where a similar approach was applied to the isoperimetric inequality, Glaudo \cite{G22} investigated the validity and the stability of Minkowski type inequalities for nearly spherical sets in $\mathbb{R}^{n+1}$ with $\rho=1$. Specifically, for a domain $\Omega\subset\mathbb{R}^{n+1}$ with nearly spherical boundary $\partial\Omega=M=\{(1+u(x),x):x\in \mathbb{S}^n\}$, where $\|u\|_{C^1}\ll1$, $\operatorname{Vol}(\Omega)=\operatorname{Vol}(\bar{B}_{1})$ and $\operatorname{bar}(\Omega)=O$, they established the following volumetric Minkowski inequality
\begin{equation}\label{Glaudo ineq}
	\int_{M}H^{+}d\mu \geq n\omega_n \left(\frac{(n+1)\operatorname{Vol}(\Omega)}{\omega_n}\right)^{\frac{n-1}{n+1}},
\end{equation}
where $H^+ = \max \{H,0\}$ denotes the positive part of the mean curvature. The first goal of this paper is to generalize this result to space forms $N^{n+1}(K)$ with a general convex weight.

\begin{theorem}\label{gH volume thm}
	Given $n \geq 3$. Let $\Omega\subset N^{n+1}(K)$ $(K=-1,0,1)$ be a domain with boundary $M=\{(\rho(1+u(x)),x):x\in\mathbb{S}^n\}$, where $\operatorname{Vol}(\Omega)=\operatorname{Vol}(\bar{B}_{\rho})$ and $\operatorname{bar}(\Omega)=O$. Let $g$ be a non-decreasing convex $C^3$ positive function. Then there exists $\varepsilon>0$ such that if  $\|u\|_{C^1}<\varepsilon$, it holds 
	\begin{equation}\label{weight gH volume ineq}
		\int_M g(\Phi) H^+ d \mu \geq \chi \left(\operatorname{Vol}(\Omega)\right),
	\end{equation}
	where $H^+ = \max (H,0)$, and $\chi$ is the unique function such that equality holds in (\ref{weight gH volume ineq}) when $\Omega$ is a geodesic ball. Equality holds in (\ref{weight gH volume ineq}) if and only if $\Omega$ is a geodesic ball centered at the origin.
	
\end{theorem}

\begin{remark}
	Notice that if $\Omega$ is $C^1$-close to a geodesic ball, then the boundary $M$ may have unbounded curvatures and may not be convex or mean convex. Under mean convex assumption, the weighted inequality (\ref{weight gH volume ineq}) specializes to several cases:
	\begin{itemize}
		\item[(1)] {For Euclidean space} ($K=0$):
		\begin{itemize}
			\item[(i)] when $g\equiv1$, it recovers the volumetric Minkowski inequality (\ref{Glaudo ineq});
			\item[(ii)] when $g(s)=s$, it corresponds to Kwong-Miao's inequality (\ref{KM ineq}).
		\end{itemize}
		
		\item[(2)] {For hyperbolic space} ($K=-1$):
		\begin{itemize}
			\item[(i)] when $g(s)=1+s$, it's the special case of $k=1$ and $\ell=-1$ in Hu-Li-Wei's inequality (\ref{HLW ineq});
			\item[(ii)] when $g(s)=s$, it's the special case of $k=1$ and $\ell=-1$ in Wei-Zhou's inequality (\ref{WZ ineq}).
		\end{itemize}
		
		\item[(3)] {For spherical case} ($K=1$):
		\begin{itemize}
			\item[(i)] when $g(s)=s$, it reduces to the inequality (\ref{sphere ineq}).
		\end{itemize}
	\end{itemize}
\end{remark}

In \cite{SX19}, Scheuer-Xia deduced the following new Minkowski type inequality in hyperbolic space.
\begin{mainthm}[\cite{SX19}]
	Let $\Omega$ be a star-shaped and mean convex domain with smooth boundary $M$ in $\mathbb{H}^{n+1}$. Then there holds
	\begin{equation}\label{SX2}
		\int_M \phi^{\prime} H d \mu \geq n(n+1) \int_{\Omega} \phi^{\prime} d v+n\omega_n^{\frac{2}{n+1}}\left((n+1) \int_{\Omega} \phi^{\prime} d v\right)^{\frac{n-1}{n+1}}.
	\end{equation}
   Equality holds in (\ref{SX2}) if and only if $\Omega$ is a geodesic ball centered at the origin.
\end{mainthm}

The second goal of this paper is to establish the inequality (\ref{SX2}) with general convex weights for nearly spherical domains in $\mathbb{H}^{n+1}$.

\begin{theorem}\label{gH weight volume thm}
	Given $n \geq 3$. Let $\Omega\subset \mathbb{H}^{n+1}$ be a domain with boundary $M=\{(\rho(1+u(x)),x):x\in\mathbb{S}^n\}$, where $\int_{\Omega}\phi' dv = \int_{\bar{B}_{\rho}}\phi' dv$ and $\operatorname{bar}(\Omega)=O$. Let $g$ be a non-decreasing convex $C^3$ positive function which satisfies $(1+s)g'(s)-g(s)\geq 0$. Then there exists $\varepsilon>0$ such that if $\|u\|_{C^1}<\varepsilon$, it holds
	\begin{equation}\label{gH weight volume ineq}
		\begin{aligned}
			\int_M g(\Phi) H^+ d \mu \geq &n \omega_n g\left(\left(1+\left(\frac{(n+1)\int_{\Omega}\phi'dv}{w_n}\right)^{\frac{2}{n+1}}\right)^{\frac{1}{2}}-1\right) \\& \times \left(\left(\frac{(n+1)\int_{\Omega}\phi'dv}{w_n}\right)^{\frac{2(n-1)}{n+1}} + \left(\frac{(n+1)\int_{\Omega}\phi'dv}{w_n}\right)^{\frac{2n}{n+1}}\right)^{\frac{1}{2}},
		\end{aligned}
	\end{equation}
	where $H^+ = \max \{H,0\}$. Equality holds in (\ref{gH weight volume ineq}) if and only if $\Omega$ is a geodesic ball centered at the origin.
	
\end{theorem}

\begin{remark}
	Recall that $\phi'=\Phi+1$ in $\mathbb{H}^{n+1}$. When taking $g(s)=1+s$, the inequality (\ref{gH weight volume ineq}) reduces to (\ref{SX2}) in the mean convex case. From the condition $(1+s)g'(s)-g(s)\geq 0$, we can deduce $g(s)=C(s)(1+s)$ for some positive and non-decreasing function $C(s)$. Theorem \ref{gH weight volume thm} generates a broad class of geometric inequalities via different convex weight selections, particularly including ${\phi'}^{\alpha} (\alpha\geq1)$ and ${\Phi}^{\alpha} (\alpha\geq1)$.
\end{remark}

For a bounded domain $\Omega$ in $\mathbb{H}^{n+1}$ with smooth boundary $M$, it is called \textit{static convex} if its second fundamental form satisfies
\begin{equation}\label{static convex}
	h_{ij}\geq\frac{\bar{\nabla}_{N} \phi'}{\phi'}g_{ij} > 0, \quad \text{everywhere on }M,
\end{equation} 
where $N$ is the outward unit normal vector on $M$. Applying locally constrained curvature flows, Hu-Li \cite{HL22} proved the following weighted geometric inequalities for static convex domains in $\mathbb{H}^{n+1}$.
\begin{mainthm}[\cite{HL22}]
	Let $\Omega$ be a static convex domain with smooth boundary $M$ in $\mathbb{H}^{n+1}$. For $0 \leq k \leq n$, there holds
	\begin{equation}\label{HL ineq}
		\int_M \phi^{\prime} \sigma_{k} d \mu \geq \binom{n}{k} \omega_n \left(\left(\frac{(n+1) \int_{\Omega} \phi^{\prime} d v}{\omega_n}\right)^{\frac{2}{k+1}}+\left(\frac{(n+1) \int_{\Omega} \phi^{\prime} d v}{\omega_n}\right)^{\frac{2(n-k)}{(n+1) (k+1)}}\right)^{\frac{k+1}{2}}.
	\end{equation}
	Equality holds in (\ref{HL ineq}) if and only if $\Omega$ is a geodesic ball centered at the origin.
	
\end{mainthm}

When $k=1$, (\ref{HL ineq}) reduces to (\ref{SX2}). Next, we investigate a generalization of (\ref{HL ineq}) with more general convex weights for nearly spherical sets of the form $M=\{(\rho(1+u(x)),x):x\in\mathbb{S}^n\}$, where $\|u\|_{W^{2,\infty}}\ll1$. Here, since higher order curvature terms are involved, we will require small bounds on $|\nabla^2 u|$ as well.

\begin{theorem}\label{gsigmak weight volume thm}
	Let $\Omega\subset \mathbb{H}^{n+1}$ be a domain with boundary $M=\{(\rho(1+u(x)),x):x\in\mathbb{S}^n\}$, where $\int_{\Omega}\phi' dv = \int_{\bar{B}_{\rho}}\phi' dv$ and $\operatorname{bar}(\Omega)=O$. Let $g$ be a non-decreasing convex $C^3$ positive function which satisfies $(1+s)g'(s)-g(s)\geq 0$. Then there exists $\varepsilon>0$ such that if $\|u\|_{W^{2,\infty}}<\varepsilon$, it holds for $0\leq k \leq n$,
	\begin{equation}\label{g sigmak weight volume ineq}
		\begin{aligned}
			\int_M g(\Phi) \sigma_{k} d \mu \geq &\binom{n}{k} \omega_n g\left(\left(1+\left(\frac{(n+1)\int_{\Omega}\phi'dv}{w_n}\right)^{\frac{2}{n+1}}\right)^{\frac{1}{2}}-1\right) \\& \times \left(\left(\frac{(n+1)\int_{\Omega}\phi'dv}{w_n}\right)^{\frac{2(n-k)}{(n+1)k}} + \left(\frac{(n+1)\int_{\Omega}\phi'dv}{w_n}\right)^{\frac{2n}{(n+1)k}}\right)^{\frac{k}{2}}.
		\end{aligned}
	\end{equation}
	Equality holds in (\ref{g sigmak weight volume ineq}) if and only if $\Omega$ is a geodesic ball centered at the origin.
	
\end{theorem}

In \cite{ZZ23}, Zhou-Zhou eatablished the stability of weighted quermassintegral inequalities by estimating the upper bound of the Fraenkel asymmetry $\alpha(\Omega)$ in $N^{n+1}(K)$, defined as
$$
\alpha(\Omega)=\inf \left\{\operatorname{Vol}\left(\Omega \Delta \bar{B}_\rho(x)\right): x \in N^{n+1}(K), \operatorname{Vol}(\Omega)=\operatorname{Vol}\left(\bar{B}_\rho(x)\right)\right\},
$$
where $\Delta$ is the symmetric difference between two sets. In the following, we prove the quantitative quermassintegral inequalities with general convex weights for nearly spherical sets in $\mathbb{H}^{n+1}$.

\begin{theorem}\label{gsgmk quermassintegral thm}
	Suppose $0\leq k \leq n-1$, $-1\leq j < k$.
    Let $\Omega\subset \mathbb{H}^{n+1}$ be a domain with boundary $M=\{(\rho(1+u(x)),x):x\in\mathbb{S}^n\}$, where $W_{j}(\Omega)=W_{j}(\bar{B}_{\rho})$ and $\operatorname{bar}(\Omega)=O$. Let $g$ be a non-decreasing convex $C^3$ positive function which satisfies $(1+s)g'(s)-g(s)\geq 0$. For any $\eta>0$, there exists $\varepsilon>0$ such that if $\|u\|_{W^{2,\infty}}<\varepsilon$, then 
	\begin{equation}\label{g sigmak weight quermassintegral ineq}
		\begin{aligned}
			&\int_M g(\Phi) \sigma_{k} d \mu - \int_{\partial\bar{B}_{\rho}} g(\Phi) \sigma_{k} d \mu\\ \geq & \left(\frac{n(n-k)(k-j)}{4\operatorname{Area}(\mathbb{S}^n)}\binom{n}{k}\frac{{\phi'}^{k}(\rho)g(\Phi(\rho))}{\phi^{n+k+2}(\rho)} - \eta\right)\alpha^2 (\Omega).
		\end{aligned}
	\end{equation}
	
\end{theorem}

\begin{remark}
	Special cases of Theorem \ref{gsgmk quermassintegral thm} include:
	\begin{itemize}
		\item[(i)] For $g(s) = 1 + s$, inequality (\ref{g sigmak weight quermassintegral ineq}) reduces to the quantitative version of (\ref{HLW ineq});
		\item[(ii)] For $g(s) = s$, it yields the quantitative version of (\ref{WZ ineq}).
	\end{itemize}
	This result extends Zhou-Zhou's work \cite[Theorem 1.6]{ZZ23}.
\end{remark}

A direct consequence of Theorem \ref{gsgmk quermassintegral thm} is the following Alexandrov-Fenchel type inequalities with general convex weights for nearly spherical sets in $\mathbb{H}^{n+1}$.

\begin{corollary}\label{gsigmak quermassintegral coro}
	Suppose $0\leq k \leq n$, $-1\leq j < k$.
	Let $\Omega\subset \mathbb{H}^{n+1}$ be a domain with boundary $M=\{(\rho(1+u(x)),x):x\in\mathbb{S}^n\}$, where $W_{j}(\Omega)=W_{j}(\bar{B}_{\rho})$ and $\operatorname{bar}(\Omega)=O$. Let $g$ be a non-decreasing convex $C^3$ positive function which satisfies $(1+s)g'(s)-g(s)\geq 0$. Then there exists $\varepsilon>0$ such that if $\|u\|_{W^{2,\infty}}<\varepsilon$, it holds 
	\begin{equation}\label{g sigmak quer ineq}
		\begin{aligned}
			\int_M g(\Phi) \sigma_{k} d \mu \geq &\zeta_{k, j}\left(W_j (\Omega)\right),
		\end{aligned}
	\end{equation}
	where $\zeta_{k, j}$ is the unique function such that equality holds in (\ref{g sigmak quer ineq}) for geodesic spheres, and equality holds in (\ref{g sigmak quer ineq}) if and only if $\Omega$ is a geodesic ball centered at the origin. 
	
\end{corollary}

	\begin{remark}
		Recent work in \cite{PL25,W25} established weighted Alexandrov-Fenchel type inequalities (\ref{g sigmak weight volume ineq}) and (\ref{g sigmak quer ineq}) for static convex hypersurfaces in $\mathbb{H}^{n+1}$ by using locally constrained inverse curvature flows. Their results required the stronger condition $g'(s) \geq \frac{k}{k-1}\frac{g(s)}{s}$, whereas our results only assume $g'(s) \geq \frac{g(s)}{1+s}$.
	\end{remark}

Finally, we establish the stability of inequality (\ref{KW gphi ineq}) for nearly spherical sets in $\mathbb{R}^{n+1}$. 


\begin{theorem}\label{thm gsk rn}
	Suppose $0\leq k \leq n$, $-1\leq j < k$. Let $\Omega\subset \mathbb{R}^{n+1}$ be a domain with boundary $M=\{(\rho(1+u(x)),x):x\in\mathbb{S}^n\}$, where $W_{j}(\Omega)=W_{j}(\bar{B}_{\rho})$ and $\operatorname{bar}(\Omega)=O$. Let $g$ be a non-decreasing convex $C^3$ positive function. For any $\eta>0$, there exists $\varepsilon>0$ such that if $\|u\|_{W^{2,\infty}}<\varepsilon$, then 
	\begin{equation}\label{stability gsk rn}
		\begin{aligned}
			&{\int_M g(\Phi)\sigma_{k}d\mu-\int_{\partial\bar{B}_{\rho}} g(\Phi)\sigma_{k}d\mu}\\ \geq& \left[\binom{n}{k}\frac{n((n-k)(k-j)g(\Phi(\rho))+(2k-j-1)\rho^{2}g'(\Phi(\rho)))}{4\operatorname{Area}(\mathbb{S}^n) \rho^{n+k+2}}-\eta\right]{\alpha}^{2}(\Omega).
		\end{aligned}
	\end{equation}

\end{theorem}

\begin{remark}
	Notice that when $g$ is a homogeneous function, i.e., $g(s)=s^{\alpha}$ $(\alpha=0\textit{~or~}\alpha\geq1)$, since the weighted curvature integrals are scaling invariant, one can restrict the stability result to the domain $\Omega$ which is close to the unit ball $\bar{B}_{1}$ and the stability is characterized using a rescaled version of the Fraenkel asymmetry, defined as 
	$$
	\bar{\alpha}(\Omega)=\inf \left\{\frac{\operatorname{Vol}\left(\Omega \Delta \bar{B}_\rho(x)\right)}{\operatorname{Vol}\left(\bar{B}_\rho (x)\right)}: x \in \mathbb{R}^{n+1}, \operatorname{Vol}(\Omega)=\operatorname{Vol}\left(\bar{B}_\rho (x)\right)\right\}.
	$$
	Theorem \ref{thm gsk rn} generalizes several special cases:
	\begin{itemize}
		\item[(i)] For $g \equiv 1$, inequality (\ref{stability gsk rn}) reduces to the quantitative Alexandrov-Fenchel inequality established by VanBlargan-Wang \cite[Theorem 1.3]{VW24};
		\item[(ii)] For $g(s) = s$, we recover the quantitative inequality proved by Zhou-Zhou \cite[Theorem 1.5]{ZZ23}.
	\end{itemize}
\end{remark}

We mention that the Fraenkel asymmetry ${\alpha}(\Omega)$ is a well-studied quantity in stability analysis. The sharp quantitative isoperimetric inequality involving the Fraenkel asymmetry, with optimal exponent $2$, was first established by Fusco-Maggi-Pratelli \cite{FMP08} through symmetrization arguments. Subsequently, Figalli-Maggi-Pratelli \cite{FMP10} extended this optimal result to the anisotropic perimeter by using mass transportation theory. For further reading on the quantitative isoperimetric inequalities and other weighted geometric inequalities in the quantitative forms, we refer to \cite{F15,FM23,GMP20,M08}.

One might conjecture that inequality (\ref{weight gH volume ineq}) holds in a $C^1$ neighborhood of the ball without any curvature assumption and without replacing $H$ with $H^{+}$. However, this is not true, as demonstrated by Glaudo \cite[Appendix A]{G22} in the following result: 

\begin{mainthm}[\cite{G22}]
	Given $n\geq2$, for any $\varepsilon>0$, there is a domain $\Omega\subset\mathbb{R}^{n+1}$ with smooth boundary $M=\{(1+u(x),x):x\in \mathbb{S}^n\}$, where $\|u\|_{C^1}<\varepsilon$, such that 
	\begin{equation*}
		\int_{M} H d\mu < -1.
	\end{equation*}
\end{mainthm}

Additionally, Chodosh-Eichmair-Koerber \cite{CEK23} recently provided a counterexample to the Minkowski inequality for closed surfaces that are $W^{2,p}\cap C^1$-close to the sphere. This shows that the Minkowski inequality fails even for perturbations of a sphere that are small in $W^{2, p} \cap C^1$ unless additional convexity assumptions are imposed. This further illustrates that the regularity condition in Theorem \ref{thm gsk rn} is optimal unless additional convexity assumptions are imposed. Specifically, they proved the following result:

\begin{mainthm}[\cite{CEK23}]
	For any $\varepsilon>0$, there is a domain $\Omega\subset\mathbb{R}^{3}$ with smooth boundary $M=\{(1+u(x),x):x\in \mathbb{S}^n\}$, where $\|u\|_{W^{2,p}\cap C^1}<\varepsilon$ for every $p\in[1,\infty)$, such that 
	\begin{equation*}
		\int_{M} H d\mu < \sqrt{16\pi |M|}.
	\end{equation*}
\end{mainthm}

The paper is organized as follows: In Section \ref{sec2}, we  review fundamental properties of elementary symmetric functions and present preliminary results for nearly spherical sets in space forms. In Section \ref{sec3}, we focus on the weighted Minkowski type inequalities in space forms and prove Theorems \ref{gH volume thm} and \ref{gH weight volume thm}. In Sections \ref{sec4} and \ref{sec5}, we study the validity and the stability of weighted quermassintegral inequalities in both $\mathbb{R}^{n+1}$ and $\mathbb{H}^{n+1}$, and prove Theorems \ref{gsigmak weight volume thm}--\ref{thm gsk rn}.

\section{Preliminaries}\label{sec2}
In this section, first let us recall some basic definitions and properties of elementary symmetric functions.

For $1\leq k \leq n$, let $\sigma_{k}$ be the $k$th elementary symmetry polynomial $\sigma_{k}:\mathbb{R}^{n}\rightarrow\mathbb{R}$ defined by
$$
\sigma_k(\lambda)=\sum_{i_1<\cdots<i_k} \lambda_{i_1} \cdots \lambda_{i_k} \text { for } \lambda=\left(\lambda_1, \cdots, \lambda_{n}\right) \in \mathbb{R}^{n}.
$$
As a convention, we take $\sigma_0 (\lambda) = 1$. For a symmetric $n \times n$ matrix $A=\left(A_i^j\right)$, we set
\begin{equation}\label{sgmk}
	\sigma_k(A)=\frac{1}{k !} \delta_{j_1 \cdots j_k}^{i_1 \cdots i_k} A_{i_1}^{j_1} \cdots A_{i_k}^{j_k},
\end{equation}
where $\delta_{j_1 \cdots j_k}^{i_1 \cdots i_k}$ is the generalized Kronecker symbol.
If $\lambda(A)=\left(\lambda_1(A), \cdots, \lambda_n(A)\right)$ are the real eigenvalues of $A$, then 
$$
\sigma_k(A)=\sigma_k(\lambda(A)) .
$$
The $k$th Newton transformation is defined as follows
$$
\left[T_k\right]_i^j(A)=\frac{\partial \sigma_{k+1}}{\partial A_j^i}(A)=\frac{1}{k !}\delta_{i i_1 \cdots i_{k}}^{j j_1 \cdots j_{k}} A_{j_1}^{i_1} \cdots A_{j_{k}}^{i_{k}}.
$$
If $A$ is a diagonal matrix with $\lambda(A)=\left(\lambda_1, \cdots, \lambda_n\right)$, then
$$
\left[T_k\right]_i^j(A) = \frac{\partial \sigma_{k+1}}{\partial \lambda_i}(\lambda)\delta_{i}^{j}.
$$

Next, we collect some well-known results of nearly spherical sets parametrized by radial function in space forms $N^{n+1}(K)$.

Let $\Omega$ be a smooth, bounded domain in $N^{n+1}(K)$ that is star-shaped with respect to the origin $O$ and enclosed by $M$. We parametrize $M$ via the radial function $M=\left\{(\rho(1+u(x)), x): x \in \mathbb{S}^n\right\}$, where $u: \mathbb{S}^n \rightarrow(-1,+\infty)$ is a smooth function.

In geodesic polar coordinates of $N^{n+1}(K)$, let $\left\{\frac{\partial}{\partial \theta_1}, \frac{\partial}{\partial \theta_2}, \ldots, \frac{\partial}{\partial \theta_n}, \frac{\partial}{\partial r}\right\}$ denote the tangent basis and $s_{i j}$ denote the canonical metric on $\mathbb{S}^n$. Then
$$
\left\langle\frac{\partial}{\partial \theta_i}, \frac{\partial}{\partial r}\right\rangle=0, \quad \left\langle\frac{\partial}{\partial r}, \frac{\partial}{\partial r}\right\rangle=1, \quad \left\langle\frac{\partial}{\partial \theta_i}, \frac{\partial}{\partial \theta_j}\right\rangle=\phi^2(r) s_{i j} .
$$
The tangent basis of $M$ is given by 
$$e_i=\frac{\partial}{\partial \theta_i}+\rho u_i \frac{\partial}{\partial r}, \quad i=1,2, \ldots, n,$$
where $u_i=\frac{\partial u}{\partial \theta_i}$. For convenience, we denote
$$
r=\rho(1+u),\quad \phi=\phi(r), \quad \phi^{\prime}=\phi^{\prime}(r), \quad D=\sqrt{\phi^2+\rho^2|\nabla u|^2},
$$
where $\nabla$ is the Levi-Civita connection on $\mathbb{S}^n$ and $|\nabla u|^2=s^{i j} u_i u_j$. Then the induced metric $g_{i j}$ on $M$ is
$$
g_{i j}=\rho^2 u_i u_j+\phi^2 s_{i j}.
$$
Thus the area element $\mathrm{d} \mu$ corresponding to the induced metric $g$ is
\begin{equation}\label{area element}
	d \mu=\sqrt{\operatorname{det}\left(g_{i j}\right)} dA=\phi^{n-1} D d A,
\end{equation}
the inverse of $\left(g_{i j}\right)$ is given by
$$
g^{i j}=\frac{s^{i j}}{\phi^2}-\frac{1}{\phi^2} \cdot \frac{\rho^2 u_k u_l s^{i k} s^{j l}}{D^2},
$$
and the outward unit normal vector $N$ on $M$ is
$$
N=\frac{-\rho s^{i j} u_i \frac{\partial}{\partial \theta_j}+\phi^2 \frac{\partial}{\partial r}}{\phi D}.
$$
Denote $h_{i j}$ as the second fundamental form on $M$. Using the definition $h_{i j}=-\left\langle\bar{\nabla}_{e_i} e_j, N\right\rangle$, then we have
$$
h_{i j}=\frac{1}{D}\left[2 \phi^{\prime} \rho^2 u_i u_j+\phi^2 \phi^{\prime} s_{i j}-\phi \rho u_{i j}\right],
$$
and the Weingarten tensor $h^i_j=g^{i k} h_{k j}$ is
$$
h_j^i=\frac{\phi^{\prime} \delta_j^i}{D}-\frac{\rho u_j^i}{D \phi}+\frac{\phi^{\prime} \rho^2 u^i u_j}{D^3}+\frac{\rho^3 u^i u_k u_j^k}{D^3 \phi},
$$
where $u^i=s^{i j} u_j, u_j^k=s^{k i} u_{i j}$. The mean curvature $H$ can be expressed as 
\begin{align}
	H & = \frac{n\phi'}{D} - \frac{\rho \Delta u}{D\phi} +\frac{\phi' \rho^2 |\nabla u|^2}{D^3} + \frac{\rho^3 \nabla^2 u \left[\nabla u, \nabla u\right]}{D^3 \phi} \label{H formula1}\\
	& = \frac{n\phi'\phi^2 + \phi'\rho^2 |\nabla u|^2}{\phi^2 D} - \frac{\rho}{\phi^2} \operatorname{div} \left(\frac{\phi}{D} \nabla u\right).\label{H formula}
\end{align}
The volume and the weighted volume of $\Omega$ are given by
	\begin{align}
		\operatorname{Vol}(\Omega) & = \int_{\mathbb{S}^n} \left(\int_{0}^{\rho(1+u)} \phi^n (r) dr\right) dA,\label{volume}\\
		\int_{\Omega} \phi^{\prime} d v & =\frac{1}{n+1} \int_{\mathbb{S}^n} \phi^{n+1}(\rho(1+u)) dA.\label{weight volume}
	\end{align}
	
\begin{lemma}
	Let $\Omega\subset N^{n+1}(K)$ $(K=-1,0,1)$ be a domain with boundary $M=\{(\rho(1+u(x)), x): x \in \mathbb{S}^n\}$. Let $g$ be a $C^3$ positive function defined on $[0,\infty)$. Suppose $\|u\|_{C^1} < \varepsilon$, then we have
		\begin{align}
			&\int_{M} g(\Phi)H d\mu \notag\\=  &\int_{\mathbb{S}^n}n\phi^{n-1}(\rho){\phi'}(\rho)g(\Phi(\rho)) dA\notag \\ &+ \int_{\mathbb{S}^n}n\left\{\left[(n-1)\phi^{n-2}(\rho){\phi'}^{2}(\rho)- K\phi^n (\rho)\right]g(\Phi(\rho)) + \phi^n (\rho)\phi'(\rho)g'(\Phi(\rho))\right\} \rho u dA \notag\\ &+ \int_{\mathbb{S}^n}{n}\left\{\left[\frac{(n-1)(n-2)}{2}\phi^{n-3}(\rho) {\phi'}^{3}(\rho) + K(1-\frac{3n}{2})\phi^{n-1} (\rho) {\phi'} (\rho)\right]g(\Phi(\rho))\right. \notag\\&\left.+ \phi^{n-1} (\rho) {\phi'} (\rho)\left[\frac{1}{2}g''(\Phi(\rho))\phi^{2}(\rho) + (n-\frac{1}{2})g'(\Phi(\rho)){\phi'} (\rho)\right]\right.\label{gPhiH int}\\&\left. - K\phi^{n+1}(\rho)g'(\Phi(\rho))\right\} \rho^2 u^2 dA\notag\\
			& + \int_{\mathbb{S}^n} \left[(n-1) \phi^{n-3}(\rho){\phi'} (\rho) g(\Phi(\rho)) + \phi^{n-1}(\rho)g'(\Phi(\rho))\right] \rho^2 |\nabla u|^2 dA \notag\\ &+ \int_{\mathbb{S}^n} \frac{\phi^{n-2} g(\Phi) \rho^3}{D^2} \nabla^2 u \left[\nabla u, \nabla u\right] dA \notag\\
			& + O(\varepsilon)\|u\|_{L^2\left(\mathbb{S}^n\right)}^2+O(\varepsilon)\|\nabla u\|_{L^2\left(\mathbb{S}^n\right)}^2.\notag
		\end{align}
	
\end{lemma}

\begin{proof}
	Starting from the expressions for the area element (\ref{area element}) and mean curvature (\ref{H formula1}), we obtain the identity
	\begin{equation}\label{gH int}
		\begin{aligned}
			\int_{M}g(\Phi)H d\mu = \int_{\mathbb{S}^n}&\left[n\phi^{n-1}\phi' g(\Phi) - \rho\phi^{n-2}g(\Phi)\Delta u + \frac{\phi^{n-1}\phi'g(\Phi)\rho^{2}}{D^2}|\nabla u|^2\right.\\&\left. + \frac{\phi^{n-2}g(\Phi)\rho^3}{D^2}\nabla^2 u \left[\nabla u, \nabla u\right]\right] dA,
		\end{aligned}
	\end{equation}
	where $r=\rho(1+u)$, $\phi=\phi(r)$, $\phi'=\phi'(r)$, $D=\sqrt{\phi^2 (r) + \rho^2 |\nabla u|^2}$ can all be seen as functions of two independent variables $u$ and $|\nabla u|^2$. 
	
	We proceed with Taylor expansion at $u=|\nabla u|^2=0$, using the assumption $\|u\|_{C^1} < \varepsilon$ to control the remainder terms and the fact $\phi''=-K\phi$, which yield that the first term expands as
		\begin{align*}
			&\int_{\mathbb{S}^n}n\phi^{n-1}{\phi'}g(\Phi) \\= &\int_{\mathbb{S}^n}n\phi^{n-1}(\rho){\phi'}(\rho)g(\Phi(\rho)) \\ &+ \int_{\mathbb{S}^n}n\left\{\left[(n-1)\phi^{n-2}(\rho){\phi'}^{2}(\rho)- K\phi^n (\rho)\right]g(\Phi(\rho)) + \phi^n (\rho)\phi'(\rho)g'(\Phi(\rho))\right\} \rho u \\ &+ \int_{\mathbb{S}^n}\frac{n}{2}\left\{\left[(n-1)(n-2)\phi^{n-3}(\rho) {\phi'}^{3}(\rho) + K(2-{3n})\phi^{n-1} (\rho) {\phi'} (\rho)\right]g(\Phi(\rho))\right. \\&\left.+ \phi^{n-1} (\rho) {\phi'} (\rho)\left[g''(\Phi(\rho))\phi^{2}(\rho) + (2n-1)g'(\Phi(\rho)){\phi'} (\rho)\right] - 2K\phi^{n+1}(\rho)g'(\Phi(\rho))\right\} \rho^2 u^2\\& + O(\varepsilon)\|u\|_{L^2\left(\mathbb{S}^n\right)}^2.
		\end{align*}
		Integrating by parts, the Laplacian term gives
		\begin{align*}
			-\int_{\mathbb{S}^n}\rho\phi^{n-2}g(\Phi)\Delta u = & \int_{\mathbb{S}^n} \left((n-2) \phi^{n-3}{\phi'}g(\Phi) + \phi^{n-1}g'(\Phi)\right) \rho^2 |\nabla u|^2
			\\= &\int_{\mathbb{S}^n}\left((n-2) \phi^{n-3}(\rho){\phi'} (\rho) g(\Phi(\rho)) + \phi^{n-1}(\rho)g'(\Phi(\rho))\right) \rho^2 |\nabla u|^2 \\&+ O(\varepsilon)\|\nabla u\|_{L^2\left(\mathbb{S}^n\right)}^2.
		\end{align*}
		The gradient term can be simplified as
		\begin{align*}
			\int_{\mathbb{S}^n}\frac{\phi^{n-1}{\phi'}g(\Phi)\rho^{2}}{D^2}|\nabla u|^2 = & \int_{\mathbb{S}^n}\phi^{n-3}(\rho){\phi'}(\rho)g(\Phi(\rho)) \rho^2 |\nabla u|^2 + O(\varepsilon)\|\nabla u\|_{L^2\left(\mathbb{S}^n\right)}^2.
		\end{align*}

	Putting the above estimates together, we get the formula (\ref{gPhiH int}).
	
\end{proof}

\begin{lemma}\label{lem expression g sigmak}
	Let $\Omega\subset N^{n+1}(K)$ $(K=-1,0,1)$ be a domain with boundary $M=\{(\rho(1+u(x)), x): x \in \mathbb{S}^n\}$. Let $g$ be a $C^3$ positive function defined on $[0,\infty)$. Suppose $\|u\|_{W^{2, \infty}}<\varepsilon$, then for any $0 \leq k \leq n$, there holds
		\begin{align}
			&\int_{M} g(\Phi)\sigma_{k}(\kappa) d\mu \label{expression g sigmak}\\ 
			=&\int_{\mathbb{S}^n} \binom{n}{k} \phi^{n-k}(\rho) {\phi'}^{k}(\rho) g(\Phi(\rho)) dA \notag\\
			&+\int_{\mathbb{S}^n} \binom{n}{k} \left\{\left[(n-k)\phi^{n-k-1}(\rho) {\phi'}^{k+1}(\rho) - Kk \phi^{n-k+1}(\rho) {\phi'}^{k-1}(\rho)\right]g(\Phi(\rho))\right. \notag\\&\left.+ \phi^{n-k+1}(\rho) {\phi'}^{k}(\rho) g'(\Phi(\rho))\right\} \rho u dA \notag\\
			&+ \int_{\mathbb{S}^n} \binom{n}{k} \left\{\left[\frac{(n-k)(n-k-1)}{2} \phi^{n-k-2}(\rho) {\phi'}^{k+2}(\rho)\right.\right. \notag\\&\left.\left.+ K\left(k^2-kn-\frac{n}{2}\right) \phi^{n-k}(\rho) {\phi'}^{k}(\rho) + K^{2}\frac{k(k-1)}{2} \phi^{n-k+2}(\rho) {\phi'}^{k-2}(\rho)\right] g(\Phi(\rho))\right. \notag\\&\left.+ \left(n-k+\frac{1}{2}\right) \phi^{n-k}(\rho) {\phi'}^{k+1}(\rho) g'(\Phi(\rho)) - Kk \phi^{n-k+2}(\rho) {\phi'}^{k-1}(\rho) g'(\Phi(\rho))\right. \notag\\&\left. + \frac{1}{2} \phi^{n-k+2}(\rho) {\phi'}^{k}(\rho) g''(\Phi(\rho))\right\} \rho^2 u^2 dA \notag\\
			&+ \int_{\mathbb{S}^n} \binom{n}{k} \left\{\left[\frac{(n-k)(k+1)}{2n} \phi^{n-k-2}(\rho) {\phi'}^{k}(\rho) - K \frac{k(k-1)}{2n} \phi^{n-k}(\rho) {\phi'}^{k-2}(\rho)\right] g(\Phi(\rho))\right. \notag\\&\left.
			 + \frac{k}{n} \phi^{n-k}(\rho) {\phi'}^{k-1}(\rho) g'(\Phi(\rho))\right\} \rho^2 |\nabla u|^2 dA \notag\\
			&+O(\varepsilon)\|u\|_{L^2\left(\mathbb{S}^n\right)}^2+O(\varepsilon)\|\nabla u\|_{L^2\left(\mathbb{S}^n\right)}^2.\notag
		\end{align}
	
\end{lemma}

\begin{proof}
	Starting from the expressions for the area element (\ref{area element}) and $\sigma_k(\kappa)$ (\cite[Lemma 3.1]{ZZ23}), we obtain the identity
	\begin{equation}\label{gphisigmak1}
		\begin{aligned}
			\int_M g(\Phi) \sigma_k(\kappa) \mathrm{d} \mu =&\int_{\mathbb{S}^n}\sum_{m=0}^k (-1)^m \binom{n-m}{k-m} \frac{{\phi'}^{k-m}g(\Phi)}{D^{k+1}} \left[ \rho^m \phi^{n-m+1} \sigma_m\left(\nabla^2 u\right)\right. \\&\left.+ \frac{k+n-2m}{n-m} \rho^{m+2} \phi^{n-m-1} u^i u_j\left[T_m\right]_i^j\left(\nabla^2 u\right)\right]\mathrm{d}A,
		\end{aligned}
	\end{equation}
	we expand $g(\Phi(\rho(1+u)))$ at $u=0$ as
	\begin{equation*}
		\begin{aligned}
			g(\Phi(\rho(1+u)))=&g(\Phi(\rho)) + g'(\Phi(\rho))\phi(\rho) \rho u \\&+ \frac{1}{2}\left[g''(\Phi(\rho))\phi^2(\rho)+g'(\Phi(\rho))\phi'(\rho)\right] \rho^2 u^{2} + o(u^2),
		\end{aligned}
	\end{equation*}
	and substitute the above formula into (\ref{gphisigmak1}), then get
	\begin{align*}
		\int_M g(\Phi) \sigma_k(\kappa) \mathrm{d} \mu
		=&\int_{\mathbb{S}^n}\sum_{m=0}^k\left[\frac{(-1)^m \rho^m \phi^{n-m+1} {\phi'}^{k-m}}{D^{k+1}} \binom{n-m}{k-m} g(\Phi(\rho)) \sigma_m\left(\nabla^2 u\right)\right.\\
		&\left.+\frac{(-1)^m \rho^{m+1} \phi^{n-m+1} {\phi'}^{k-m} u }{D^{k+1}} \binom{n-m}{k-m} g'(\Phi(\rho))\phi(\rho) \sigma_m\left(\nabla^2 u\right)\right.\\
		&\left.+\frac{(-1)^m \rho^{m+2} \phi^{n-m+1} {\phi'}^{k-m} u^2 }{D^{k+1}} \binom{n-m}{k-m}\right. \\ &\left.\times \frac{1}{2}\left(g''(\Phi(\rho))\phi^2(\rho)+g'(\Phi(\rho))\phi'(\rho)\right) \sigma_m\left(\nabla^2 u\right)\right.\\
		&\left.+\frac{(-1)^m \rho^{m+2} \phi^{n-m-1} {\phi'}^{k-m} g(\Phi) }{D^{k+1}}\binom{n-m}{k-m}\right. \\
		&\left.\times\frac{k+n-2 m}{n-m} u^i u_j\left[T_m\right]_i^j\left(\nabla^2 u\right)\right]\mathrm{d}A\\
		&+O(\varepsilon)\|u\|_{L^2\left(\mathbb{S}^n\right)}^2+O(\varepsilon)\|\nabla u\|_{L^2\left(\mathbb{S}^n\right)}^2.
	\end{align*}
	
	We proceed above equation with Taylor expansion at $u=|\nabla u|^2=0$, using the assumption $\|u\|_{W^{2,\infty}} < \varepsilon$ to control the remainder terms. This yields 
	\begin{equation}\label{gsigmak111}
		\begin{aligned}
			\int_M g(\Phi) \sigma_k(\kappa) \mathrm{d} \mu
			=&\int_{\mathbb{S}^n}\sum_{m=0}^k\left[\left(A_{0}^{m}+A_{1}^{m}u+A_{2}^{m}u^2+A^{m}|\nabla u|^2\right) g(\Phi(\rho)) \sigma_m\left(\nabla^2 u\right)\right.\\
			&\left.+\left(B_{1}^{m}u+B_{2}^{m}u^2\right) g'(\Phi(\rho))\phi(\rho) \sigma_m\left(\nabla^2 u\right)\right.\\
			&\left.+C_{2}^{m}u^2 \frac{1}{2} \left(g''(\Phi(\rho))\phi^2(\rho)+g'(\Phi(\rho))\phi'(\rho)\right) \sigma_m\left(\nabla^2 u\right)\right.\\
			&\left.+D^m u^i u_j\left[T_m\right]_i^j\left(\nabla^2 u\right)\right]\mathrm{d}A\\
			&+O(\varepsilon)\|u\|_{L^2\left(\mathbb{S}^n\right)}^2+O(\varepsilon)\|\nabla u\|_{L^2\left(\mathbb{S}^n\right)}^2,
		\end{aligned}
	\end{equation} 
	where for $i=0,1,2,$
	\begin{align}
		A_{i}^{m}=&(i!)^{-1}\left.\frac{\partial^{i}}{\partial^{i}u}\left[\frac{(-1)^m \rho^m \phi^{n-m+1} {\phi'}^{k-m}}{D^{k+1}} \binom{n-m}{k-m}\right]\right|_{u=|\nabla u|^2=0},\label{aim}\\
		A^{m}=&\left.\frac{\partial}{\partial|\nabla u|^2}\left[\frac{(-1)^m \rho^m \phi^{n-m+1} {\phi'}^{k-m}}{D^{k+1}} \binom{n-m}{k-m}\right]\right|_{u=|\nabla u|^2=0},\\
		B_{i}^{m}=&(i!)^{-1}\left.\frac{\partial^{i}}{\partial^{i}u}\left[\frac{(-1)^m \rho^{m+1} \phi^{n-m+1} {\phi'}^{k-m} u }{D^{k+1}} \binom{n-m}{k-m}\right]\right|_{u=|\nabla u|^2=0},\\
		C_{i}^{m}=&(i!)^{-1}\left.\frac{\partial^{i}}{\partial^{i}u}\left[\frac{(-1)^m \rho^{m+2} \phi^{n-m+1} {\phi'}^{k-m} u^2 }{D^{k+1}} \binom{n-m}{k-m}\right]\right|_{u=|\nabla u|^2=0},\\
		D^{m}=&\left.\frac{(-1)^m \rho^{m+2} \phi^{n-m-1} {\phi'}^{k-m} g(\Phi) }{D^{k+1}}\binom{n-m}{k-m}\frac{k+n-2 m}{n-m}\right|_{u=|\nabla u|^2=0}\nonumber\\
		=&(-1)^m \rho^{m+2} \phi^{n-m-k-2}(\rho) {\phi'}^{k-m}(\rho) g(\Phi(\rho)) \binom{n-m}{k-m}\frac{k+n-2 m}{n-m}.\label{dm}
	\end{align}
	
	Using the following identities which have been proved in \cite[Lemma 4.2]{VW24}: 
	\begin{align}
		\int_{\mathbb{S}^n} u^i u_j\left[T_m\right]_i^j\left(\nabla^2 u\right) \mathrm{d} A= & \frac{m+2}{2} \int_{\mathbb{S}^n}|\nabla u|^2 \sigma_m\left(\nabla^2 u\right) \mathrm{d} A \nonumber\\
		& +O(\varepsilon)\|\nabla u\|_{L^2\left(\mathbb{S}^n\right)}^2 \quad(m \geqslant 1), \\
		\int_{\mathbb{S}^n} \sigma_m\left(\nabla^2 u\right) \mathrm{d} A= & \frac{n-m+1}{2} \int_{\mathbb{S}^n}|\nabla u|^2 \sigma_{m-2}\left(\nabla^2 u\right) \mathrm{d} A \nonumber\\
		& +O(\varepsilon)\|\nabla u\|_{L^2\left(\mathbb{S}^n\right)}^2 \quad(m \geqslant 2), \\
		\int_{\mathbb{S}^n} \sigma_1\left(\nabla^2 u\right) \mathrm{d} A= & 0, \\
		\int_{\mathbb{S}^n} u \sigma_m\left(\nabla^2 u\right) \mathrm{d} A=&-  \frac{m+1}{2 m} \int_{\mathbb{S}^n}|\nabla u|^2 \sigma_{m-1}\left(\nabla^2 u\right) \mathrm{d} A \nonumber\\
		& +O(\varepsilon)\|\nabla u\|_{L^2\left(\mathbb{S}^n\right)}^2 \quad(m \geqslant 1), \\
		\int_{\mathbb{S}^n} u^2 \sigma_m\left(\nabla^2 u\right) \mathrm{d} A= & O(\varepsilon)\|\nabla u\|_{L^2\left(\mathbb{S}^n\right)}^2 \quad(m \geqslant 1),
	\end{align}
	we can rewrite (\ref{gsigmak111}) as
	\begin{align*}
		&\int_M g(\Phi) \sigma_k(\kappa) \mathrm{d} \mu \\
		=&\int_{\mathbb{S}^n}\left[A_{0}^{0}g(\Phi(\rho))+\left(A_{1}^{0}g(\Phi(\rho))+B_{1}^{0}g'(\Phi(\rho))\phi(\rho)\right)u+\left(A_{2}^{0}g(\Phi(\rho))+B_{2}^{0}g'(\Phi(\rho))\phi(\rho)\right.\right.\\&\left.\left.+C_{2}^{0}\frac{1}{2}\left(g''(\Phi(\rho))\phi^2(\rho)+g'(\Phi(\rho))\phi'(\rho)\right)\right)u^2+\left(A^{0}g(\Phi(\rho))+D^0\right)|\nabla u|^2\right]\mathrm{d}A\\
		&+\int_{\mathbb{S}^n}\left[-\left(A_{1}^{1}g(\Phi(\rho))+B_{1}^{1}g'(\Phi(\rho))\phi(\rho)\right)|\nabla u|^2\right.\\&\left.+\left(A^{1}g(\Phi(\rho))+\frac{3}{2}D^1\right)|\nabla u|^2\sigma_1\left(\nabla^2 u\right)\right]\mathrm{d}A\\
		&+\int_{\mathbb{S}^n}\sum_{m=2}^{k}\left[\frac{n-m+1}{2}|\nabla u|^2\sigma_{m-2}\left(\nabla^2 u\right)A_{0}^{m}g(\Phi(\rho))\right.\\&\left.-\left(A_{1}^{m}g(\Phi(\rho))+B_{1}^{m}g'(\Phi(\rho))\phi(\rho)\right)\frac{m+1}{2m}|\nabla u|^2\sigma_{m-1}\left(\nabla^2 u\right)\right.\\&\left.+A^{m}g(\Phi(\rho))|\nabla u|^2\sigma_m\left(\nabla^2 u\right)+D^m\frac{m+2}{2}|\nabla u|^2\sigma_m\left(\nabla^2 u\right)\right]\mathrm{d}A\\&+O(\varepsilon)\|u\|_{L^2\left(\mathbb{S}^n\right)}^2+O(\varepsilon)\|\nabla u\|_{L^2\left(\mathbb{S}^n\right)}^2\\
		=&\int_{\mathbb{S}^n}\left[A_{0}^{0}g(\Phi(\rho))+\left(A_{1}^{0}g(\Phi(\rho))+B_{1}^{0}g'(\Phi(\rho))\phi(\rho)\right)u\right.\\&\left.+\left(A_{2}^{0}g(\Phi(\rho))+B_{2}^{0}g'(\Phi(\rho))\phi(\rho)+C_{2}^{0}\frac{1}{2}\left(g''(\Phi(\rho))\phi^2(\rho)+g'(\Phi(\rho))\phi'(\rho)\right)\right)u^2\right]\mathrm{d}A\\
		&+\sum_{m=0}^{k-2}\int_{\mathbb{S}^n}\left[A^{m}g(\Phi(\rho))+D^m\frac{m+2}{2}-\left(A_{1}^{m+1}g(\Phi(\rho))+B_{1}^{m+1}g'(\Phi(\rho))\phi(\rho)\right)\frac{m+2}{2(m+1)}\right.\\&\left.+A_{0}^{m+2}g(\Phi(\rho))\frac{n-m-1}{2}\right]|\nabla u|^2\sigma_m\left(\nabla^2 u\right)\mathrm{d}A\\
		&+\int_{\mathbb{S}^n}\left[A^{k-1}g(\Phi(\rho))+D^{k-1}\frac{k+1}{2}-\left(A_{1}^{k}g(\Phi(\rho))+B_{1}^{k}g'(\Phi(\rho))\phi(\rho)\right)\frac{k+1}{2k}\right]\\&\times|\nabla u|^2\sigma_{k-1}\left(\nabla^2 u\right)\mathrm{d}A\\
		&+\int_{\mathbb{S}^n}\left[A^{k}g(\Phi(\rho))+D^{k}\frac{k+2}{2}\right]|\nabla u|^2\sigma_{k}\left(\nabla^2 u\right)\mathrm{d}A\\
		&+O(\varepsilon)\|u\|_{L^2\left(\mathbb{S}^n\right)}^2+O(\varepsilon)\|\nabla u\|_{L^2\left(\mathbb{S}^n\right)}^2.
	\end{align*}
	Note that for $1\leq k\leq n$,
	\begin{equation*}
		\sum_{m=1}^{k}\int_{\mathbb{S}^n}C(n,m,k,\rho)|\nabla u|^2\sigma_{m}\left(D^2 u\right)\mathrm{d}A=O(\varepsilon)\|\nabla u\|_{L^2\left(\mathbb{S}^n\right)}^2,
	\end{equation*}
	we have
	\begin{align*}
		&\int_M g(\Phi) \sigma_k(\kappa) \mathrm{d} \mu \\
		=&\int_{\mathbb{S}^n}\left[A_{0}^{0}g(\Phi(\rho))+\left(A_{1}^{0}g(\Phi(\rho))+B_{1}^{0}g'(\Phi(\rho))\phi(\rho)\right)u\right.\\&\left.+\left(A_{2}^{0}g(\Phi(\rho))+B_{2}^{0}g'(\Phi(\rho))\phi(\rho)+C_{2}^{0}\frac{1}{2}\left(g''(\Phi(\rho))\phi^2(\rho)+g'(\Phi(\rho))\phi'(\rho)\right)\right)u^2\right]\mathrm{d}A\\
		&+\int_{\mathbb{S}^n}\left[A^{0}g(\Phi(\rho))+D^0-\left(A_{1}^{1}g(\Phi(\rho))+B_{1}^{1}g'(\Phi(\rho))\phi(\rho)\right)+A_{0}^{2}g(\Phi(\rho))\frac{n-1}{2}\right]\\&\times|\nabla u|^2\mathrm{d}A+O(\varepsilon)\|u\|_{L^2\left(\mathbb{S}^n\right)}^2+O(\varepsilon)\|\nabla u\|_{L^2\left(\mathbb{S}^n\right)}^2.
	\end{align*}
	Then from (\ref{aim})--(\ref{dm}), using the fact ${\phi'}^2 + K\phi^2 =1$ and direct calculation, we get the formula (\ref{expression g sigmak}).

\end{proof}

\section{Weighted Minkowski type inequalities in space forms}\label{sec3}
In this section, we will use the similar idea as in the work of Glaudo \cite{G22} to prove Theorems \ref{gH volume thm} and \ref{gH weight volume thm}. Before that, the following Poincar\'{e}-type estimate is crucial for our proofs.

\begin{lemma}[Poincar\'{e}-type Estimate]
	Let $M=\{(\rho(1+u(x)),x):x\in\mathbb{S}^n\}$ be a nearly spherical set with $\|u\|_{C^1}<\varepsilon$ for some small $\varepsilon>0$. Let $\Omega$ be the domain enclosed by $M$ in $N^{n+1}(K)$ $(K=-1,0,1)$ satisfying $\operatorname{bar}(\Omega)=O$.
	\begin{itemize}
		\item[(i)] If $\operatorname{Vol}(\Omega)=\operatorname{Vol}(\bar{B}_{\rho})$, then 
		\begin{equation}\label{volume poin esti}
			\|\nabla u\|_{L^2\left(\mathbb{S}^n\right)}^2 \geq 2(n+1)\|u\|_{L^2\left(\mathbb{S}^n\right)}^2 + O(\varepsilon)\|u\|_{L^2\left(\mathbb{S}^n\right)}^2.
		\end{equation}
		\item[(ii)] If $\int_{\Omega}\phi' dv = \int_{\bar{B}_{\rho}}\phi' dv$, then
		\begin{equation}\label{weight volume poin esti}
			\|\nabla u\|_{L^2\left(\mathbb{S}^n\right)}^2 \geq 2(n+1)\|u\|_{L^2\left(\mathbb{S}^n\right)}^2 + O(\varepsilon)\|u\|_{L^2\left(\mathbb{S}^n\right)}^2.
		\end{equation}
	\end{itemize}
\end{lemma}
\begin{proof}
	Suppose $u$ can be represented as 
	$$
	u=\sum_{k=0}^{\infty} a_k Y_k,
	$$
	where $\left\{Y_k\right\}_{k=0}^{\infty}$ corresponds the spherical harmonics which forms an orthonormal basis for $L^2 (\mathbb{S}^n)$. Since $Y_0 = 1$, we have
	$$
	a_0 = \int_{\mathbb{S}^n} u dA.
	$$
	
	(i) Assume that $\operatorname{Vol}(\Omega)=\operatorname{Vol}(\bar{B}_{\rho})$, by (\ref{volume}) we have
	\begin{equation*}
		\begin{aligned}
			0=&\operatorname{Vol}(\Omega)-\operatorname{Vol}(\bar{B}_{\rho})\\
			=&\int_{\mathbb{S}^n}\left(\int_{\rho}^{\rho(1+u)}\phi^n(r)dr\right)dA\\
			=&\int_{\mathbb{S}^n}\phi^n(\rho)\rho udA + \int_{\mathbb{S}^n}\frac{n}{2}\phi^{n-1}(\rho)\phi'(\rho)\rho^2 u^2 dA + O(\varepsilon)\|u\|_{L^2\left(\mathbb{S}^n\right)}^2.
		\end{aligned}
	\end{equation*}
	It implies that
	\begin{equation}\label{u volume}
		\int_{\mathbb{S}^n}udA=-\frac{n}{2}\frac{\phi'(\rho)}{\phi(\rho)}\rho	\int_{\mathbb{S}^n}u^2 dA+ O(\varepsilon)\|u\|_{L^2\left(\mathbb{S}^n\right)}^2.
	\end{equation}
	Therefore, we get
	\begin{equation}\label{volume a0}
		a_0 = O(\varepsilon)\|u\|_{L^2\left(\mathbb{S}^n\right)}.
	\end{equation}
	
	(ii) Assume that $\int_{\Omega}\phi' dv = \int_{\bar{B}_{\rho}}\phi' dv$, using (\ref{weight volume}) and $\phi''=-K\phi$, we have
	\begin{equation*}
		\begin{aligned}
			0=&(n+1)\int_{\Omega}\phi' dv - (n+1)\int_{\bar{B}_{\rho}}\phi' dv\\
			=&\int_{\mathbb{S}^n} (n+1) \phi^{n}(\rho) \phi'(\rho) \rho udA + \int_{\mathbb{S}^n} \frac{n+1}{2} \left[n\phi^{n-1}(\rho){\phi'}^2 (\rho) - K \phi^{n+1}(\rho)\right]\rho^2 u^2 dA \\
			&+ O(\varepsilon)\|u\|_{L^2\left(\mathbb{S}^n\right)}^2.
		\end{aligned}
	\end{equation*}
	It implies that
	\begin{equation}\label{u weight volum}
		\int_{\mathbb{S}^n}udA=-\frac{1}{2} \left[n\frac{\phi'(\rho)}{\phi(\rho)} - K\frac{\phi(\rho)}{\phi'(\rho)}\right]\rho \int_{\mathbb{S}^n}u^2 dA+ O(\varepsilon)\|u\|_{L^2\left(\mathbb{S}^n\right)}^2.
	\end{equation}
	Thus we have
	\begin{equation}\label{weighted volume a0}
		a_0 = O(\varepsilon)\|u\|_{L^2\left(\mathbb{S}^n\right)}.
	\end{equation}
	
	As shown in \cite[Lemma 5.2]{VW24} or \cite[Lemma 4.1]{ZZ23}, under the assumption $\operatorname{bar}(\Omega)=O$, then
	\begin{equation}\label{a1}
		a_1 = O(\varepsilon)\|u\|_{L^2\left(\mathbb{S}^n\right)}.
	\end{equation}
	Note that the eigenvalues of the Laplace-Beltrami operator $-\Delta$ on $\mathbb{S}^n$ are $\lambda_{k}=k(n+k-1)$ $(k=0,1,\cdots)$. When $k \geq 2$, we have that $\lambda_k \geq 2(n+1)$, then
	\begin{equation*}
		\begin{aligned}
			\|\nabla u\|_{L^2\left(\mathbb{S}^n\right)}^2 =& \sum_{k=1}^{\infty} \lambda_k a_k^2 = \sum_{k=2}^{\infty} \lambda_k a_k^2 + n a_1^2 \\
			\geq& 2(n+1) \sum_{k=2}^{\infty} a_k^2 + n a_1^2 \\
			=& 2(n+1) \sum_{k=0}^{\infty} a_k^2 - 2(n+1) a_0^2 - (n+2) a_1^2 \\
			=& 2(n+1) \|u\|_{L^2\left(\mathbb{S}^n\right)}^2 + O(\varepsilon)\|u\|_{L^2\left(\mathbb{S}^n\right)}^2.
		\end{aligned}
	\end{equation*}
	The last equality is given by (\ref{volume a0}), (\ref{weighted volume a0}) and (\ref{a1}).
	
\end{proof}

\begin{proof}[Proof of Theorem \ref{gH volume thm}]
	Assume that $\operatorname{Vol}(\Omega)=\operatorname{Vol}(\bar{B}_{\rho})$, then
	plugging (\ref{u volume}) into (\ref{gPhiH int}) and using ${\phi'}^2 + K\phi^2 =1$ we have
	\begin{equation}
		\begin{aligned}
			&\int_{M} g(\Phi) H d\mu - \int_{\partial\bar{B}_{\rho}} g(\Phi) H d\mu \\=
			& - \int_{\mathbb{S}^n} n\left[(n-1)\phi^{n-3}(\rho) {\phi'}(\rho) g(\Phi(\rho))-\frac{1}{2}\phi^{n+1}(\rho) {\phi'}(\rho) g''(\Phi(\rho))\right.\\&\left.-\frac{n+1}{2}\phi^{n-1}(\rho) {\phi'}^2(\rho) g'(\Phi(\rho))+\phi^{n-1}(\rho) g'(\Phi(\rho))\right] \rho^2 u^2 dA \\&+ \int_{\mathbb{S}^n} \left[(n-1)\phi^{n-3}(\rho)\phi'(\rho) g(\Phi(\rho)) + \phi^{n-1}(\rho) g'(\Phi(\rho))\right] \rho^2 |\nabla u|^2 dA \\ & + \int_{\mathbb{S}^n} \frac{\phi^{n-2} g(\Phi) \rho^3}{D^2} \nabla^2 u \left[\nabla u, \nabla u\right] dA + O(\varepsilon)\|u\|_{L^2\left(\mathbb{S}^n\right)}^2+O(\varepsilon)\|\nabla u\|_{L^2\left(\mathbb{S}^n\right)}^2.
		\end{aligned}
	\end{equation}
	Recall the assumption that $g$ is a non-decreasing convex $C^3$ positive function, we have
		\begin{align}
			&\int_{M} g(\Phi) H d\mu - \int_{\partial\bar{B}_{\rho}} g(\Phi) H d\mu \notag\\ \geq
			& - \int_{\mathbb{S}^n} n\left[(n-1)\phi^{n-3}(\rho) {\phi'}(\rho) g(\Phi(\rho))+\phi^{n-1}(\rho) g'(\Phi(\rho))\right] \rho^2 u^2 dA \label{gH int volume}\\&+ \int_{\mathbb{S}^n} \left[(n-1)\phi^{n-3}(\rho)\phi'(\rho) g(\Phi(\rho)) + \phi^{n-1}(\rho) g'(\Phi(\rho))\right] \rho^2 |\nabla u|^2 dA \notag\\ & + \int_{\mathbb{S}^n} \frac{\phi^{n-2} g(\Phi) \rho^3}{D^2} \nabla^2 u \left[\nabla u, \nabla u\right] dA + O(\varepsilon)\|u\|_{L^2\left(\mathbb{S}^n\right)}^2+O(\varepsilon)\|\nabla u\|_{L^2\left(\mathbb{S}^n\right)}^2.\notag
		\end{align}
	
	Denote $H^{-} := -\min(H,0)$, thanks to (\ref{H formula}), we have
	\begin{equation}\label{div estimate1}
		\begin{aligned}
			\operatorname{div}\left(\frac{\phi}{D}\nabla u\right) &= \frac{n\phi'\phi^2 + \phi'\rho^2 |\nabla u|^2}{\rho D} - \frac{\phi^2}{\rho}H \\
			&= n\frac{\phi'(\rho)\phi(\rho)}{\rho} + O(\varepsilon) - H\left(\frac{\phi^2 (\rho)}{\rho} + O(\varepsilon)\right) \\
			&\leq n\frac{\phi'(\rho)\phi(\rho)}{\rho} + O(\varepsilon) + 2\frac{\phi^2 (\rho)}{\rho}H^{-}.
		\end{aligned}
	\end{equation}
	Let $f(u,|\nabla u|^2) :=\frac{\phi^{n-3}g(\Phi)}{D}$ and $F(u,|\nabla u|^2) := \int_{0}^{|\nabla u|^2}f(u,t) dt$. One can check that $\partial_{1}F(u,|\nabla u|^2)=O(\varepsilon)$, $\partial_{2}F(u,|\nabla u|^2)=f(u,|\nabla u|^2)$ and $F(u,|\nabla u|^2) = \phi^{n-4}(\rho)g(\Phi(\rho))(1+O(\varepsilon)) |\nabla u|^2$. It implies
	\begin{equation*}
		\nabla\left( F \left(u,|\nabla u|^2\right)\right) = \frac{\phi^{n-3}g(\Phi)}{D} \nabla |\nabla u|^2 + O(\varepsilon) \nabla u.
	\end{equation*}
	Integrating by parts, the cubic term appearing in (\ref{gH int volume}) becomes
	\begin{equation}\label{H second term1}
		\begin{aligned}
			\int_{\mathbb{S}^n} \frac{\phi^{n-2} g(\Phi) \rho^3}{D^2} \nabla^2 u \left[\nabla u, \nabla u\right] =& \frac{\rho^3}{2} \int_{\mathbb{S}^n} \frac{\phi}{D}\nabla u \cdot \frac{\phi^{n-3} g(\Phi)}{D}\nabla |\nabla u|^2 \\
			=& \frac{\rho^3}{2} \int_{\mathbb{S}^n} \frac{\phi}{D}\nabla u \cdot \nabla\left( F \left(u,|\nabla u|^2\right)\right) + O(\varepsilon)\|\nabla u\|_{L^2\left(\mathbb{S}^n\right)}^2 \\
			=& -\frac{\rho^3}{2} \int_{\mathbb{S}^n} \operatorname{div}\left(\frac{\phi}{D}\nabla u\right)F \left(u,|\nabla u|^2\right) + O(\varepsilon)\|\nabla u\|_{L^2\left(\mathbb{S}^n\right)}^2 \\
			=& -\frac{1}{2}\phi^{n-4}(\rho) g(\Phi(\rho)) \rho^3 (1+O(\varepsilon)) \int_{\mathbb{S}^n} \operatorname{div}\left(\frac{\phi}{D}\nabla u\right)|\nabla u|^2 \\
			&+ O(\varepsilon)\|\nabla u\|_{L^2\left(\mathbb{S}^n\right)}^2.
		\end{aligned}
	\end{equation}
	
	Decomposing $u=u_1 + u_2$, where $u_1, u_2$ belong to the subspaces generated by the eigenfunctions of $-\Delta$ with eigenvalues respectively smaller and larger than a fixed $\lambda>0$. Notice that $u_1$ lives in the finite-dimensional space and all norms are equivalent on a finite-dimensional space. Therefore, $\|\nabla^2 u_1\|_{\infty} \leq c(n)\|\nabla u_1\|_{\infty} = O(\varepsilon)$, where $c(n) > 0$ is a constant that depends only on $n$. 
	
	Let $h(u,|\nabla u|^2) :=\frac{\phi}{D}$ and $H(u,|\nabla u|^2) := \int_{0}^{|\nabla u|^2}h(u,t) dt$. Notice that $\partial_{1}H(u,|\nabla u|^2)=O(\varepsilon)$, $\partial_{2}H(u,|\nabla u|^2)=h(u,|\nabla u|^2)$ and $H(u,|\nabla u|^2) = (1+O(\varepsilon)) |\nabla u|^2$. It implies
	\begin{equation*}
		\nabla\left( H \left(u,|\nabla u|^2\right)\right) = \frac{\phi}{D} \nabla |\nabla u|^2 + O(\varepsilon) \nabla u.
	\end{equation*}
	Integrating by parts, we get
	\begin{equation}\label{decomposing div1}
		\begin{aligned}
			\int_{\mathbb{S}^n} \operatorname{div}\left(\frac{\phi}{D}\nabla u\right)|\nabla u|^2 =& -2\int_{\mathbb{S}^n} \frac{\phi}{D} \nabla^2 u \left[\nabla u, \nabla u\right] \\
			=& -2\int_{\mathbb{S}^n} \frac{\phi}{D} \nabla^2 u \left[\nabla u, \nabla u_1\right] -2\int_{\mathbb{S}^n} \frac{\phi}{D} \nabla^2 u_2 \left[\nabla u, \nabla u_2\right] \\
			&+ O(\varepsilon)\|\nabla u\|_{L^2\left(\mathbb{S}^n\right)}^2 \\
			=& -\int_{\mathbb{S}^n} \nabla\left( H \left(u,|\nabla u|^2\right)\right) \cdot \nabla u_1 -\int_{\mathbb{S}^n} \frac{\phi}{D} \nabla u \cdot \nabla |\nabla u_2|^2 \\
			&+ O(\varepsilon)\|\nabla u\|_{L^2\left(\mathbb{S}^n\right)}^2\\
			=& \int_{\mathbb{S}^n} H \left(u,|\nabla u|^2\right) \cdot \Delta u_1 + \int_{\mathbb{S}^n} \operatorname{div}\left(\frac{\phi}{D}\nabla u\right)|\nabla u_2|^2 \\
			&+ O(\varepsilon)\|\nabla u\|_{L^2\left(\mathbb{S}^n\right)}^2\\
			=& \int_{\mathbb{S}^n} \operatorname{div}\left(\frac{\phi}{D}\nabla u\right)|\nabla u_2|^2 + O(\varepsilon)\|\nabla u\|_{L^2\left(\mathbb{S}^n\right)}^2.
		\end{aligned}
	\end{equation}
	Applying (\ref{div estimate1}) and (\ref{decomposing div1}), (\ref{H second term1}) becomes
	\begin{equation*}
		\begin{aligned}
			\int_{\mathbb{S}^n} \frac{\phi^{n-2} g(\Phi) \rho^3}{D^2} \nabla^2 u \left[\nabla u, \nabla u\right] =& -\frac{1}{2}\phi^{n-4}(\rho) g(\Phi(\rho)) \rho^3 (1+O(\varepsilon)) \int_{\mathbb{S}^n} \operatorname{div}\left(\frac{\phi}{D}\nabla u\right)|\nabla u_2|^2 \\
			&+ O(\varepsilon)\|\nabla u\|_{L^2\left(\mathbb{S}^n\right)}^2 \\
			\geq & -\int_{\mathbb{S}^n} \frac{n}{2}\phi^{n-3}(\rho) {\phi'}(\rho) g(\Phi(\rho)) \rho^2 |\nabla u_2|^2 + O(\varepsilon) \int_{\mathbb{S}^n} H^{-} \\
			& + O(\varepsilon)\|\nabla u\|_{L^2\left(\mathbb{S}^n\right)}^2.
		\end{aligned}
	\end{equation*}
	Notice that by (\ref{area element}), we have
	\begin{equation*}
		\begin{aligned}
			\int_{M}g(\Phi)H^{-} =& \int_{\mathbb{S}^n}g(\Phi)H^{-}\phi^{n-1}D = \int_{\mathbb{S}^n}H^{-}\phi^{n}(\rho)g(\Phi(\rho))(1+O(\varepsilon))\\ \geq &\frac{\phi^{n}(\rho)g(\Phi(\rho))}{2} \int_{\mathbb{S}^n}H^{-}.
		\end{aligned}
	\end{equation*}
To sum the above two equations, we can obtain
	\begin{equation}\label{estimate}
		\begin{aligned}
			\int_{\mathbb{S}^n} \frac{\phi^{n-2} g(\Phi) \rho^3}{D^2} \nabla^2 u \left[\nabla u, \nabla u\right] + \int_{M}g(\Phi)H^{-} \geq & -\int_{\mathbb{S}^n} \frac{n}{2}\phi^{n-3}(\rho) {\phi'}(\rho) g(\Phi(\rho)) \rho^2 |\nabla u_2|^2 \\&+ O(\varepsilon)\|\nabla u\|_{L^2\left(\mathbb{S}^n\right)}^2.
		\end{aligned}
	\end{equation}
	Putting the last inequality and (\ref{gH int volume}) togather, we have
	\begin{equation}\label{defict g H int}
		\begin{aligned}
			&\int_{M} g(\Phi) H^+ d\mu - \int_{\partial\bar{B}_{\rho}} g(\Phi) H d\mu \\\geq
			& \int_{\mathbb{S}^n} \left[(n-1)\phi^{n-3}(\rho)\phi'(\rho) g(\Phi(\rho)) + \phi^{n-1}(\rho) g'(\Phi(\rho))\right] \rho^2 |\nabla u_1|^2 dA \\&- \int_{\mathbb{S}^n} n\left[(n-1)\phi^{n-3}(\rho) {\phi'}(\rho) g(\Phi(\rho))+\phi^{n-1}(\rho) g'(\Phi(\rho))\right] \rho^2 u_1^2 dA \\&+ \int_{\mathbb{S}^n} \left[\frac{n-2}{2}\phi^{n-3}(\rho)\phi'(\rho) g(\Phi(\rho)) + \phi^{n-1}(\rho) g'(\Phi(\rho))\right] \rho^2 |\nabla u_2|^2 dA \\ &  - \int_{\mathbb{S}^n} n\left[(n-1)\phi^{n-3}(\rho) {\phi'}(\rho) g(\Phi(\rho))+\phi^{n-1}(\rho) g'(\Phi(\rho))\right] \rho^2 u_2^2 dA  \\&+ O(\varepsilon)\|u\|_{L^2\left(\mathbb{S}^n\right)}^2+O(\varepsilon)\|\nabla u\|_{L^2\left(\mathbb{S}^n\right)}^2.
		\end{aligned}
	\end{equation}
	
	Since $u_2$ belongs to the subspace generated by eigenfunctions of the Laplace-Beltrami operator with eigenvalues larger than $\lambda$, it holds
	\begin{equation}\label{u2 volume poin esti1}
		\|\nabla u_2\|_{L^2\left(\mathbb{S}^n\right)}^2 \geq \lambda \| u_2\|_{L^2\left(\mathbb{S}^n\right)}^2.
	\end{equation}
	In addition, since $u_1$ is the projection of $u$ onto low frequencies, it inherits the estimates of (\ref{volume a0}) and (\ref{a1}). Repeating the argument used to establish (\ref{volume poin esti}), we can also prove
	\begin{equation}\label{u1 volume poin esti1}
		\|\nabla u_1\|_{L^2\left(\mathbb{S}^n\right)}^2 \geq 2(n+1)\|u_1\|_{L^2\left(\mathbb{S}^n\right)}^2 + O(\varepsilon)\|u\|_{L^2\left(\mathbb{S}^n\right)}^2.
	\end{equation}
	
	Using the two Poincar\'{e}-type inequalities (\ref{u2 volume poin esti1}) and (\ref{u1 volume poin esti1}), the estimate (\ref{defict g H int}) becomes 
	\begin{equation}\label{defict g H int1}
			\begin{aligned}
			&\int_{M} g(\Phi) H^+ d\mu - \int_{\partial\bar{B}_{\rho}} g(\Phi) H d\mu \\\geq
			& \int_{\mathbb{S}^n} \frac{n+2}{2(n+1)}\left[(n-1)\phi^{n-3}(\rho)\phi'(\rho) g(\Phi(\rho)) + \phi^{n-1}(\rho) g'(\Phi(\rho))\right] \rho^2 |\nabla u_1|^2 dA \\&+ \int_{\mathbb{S}^n} \left[\left(\frac{n-2}{2}-\frac{n(n-1)}{\lambda}\right)\phi^{n-3}(\rho)\phi'(\rho) g(\Phi(\rho))\right. \\&\left.+ \left(1-\frac{n}{\lambda}\right)\phi^{n-1}(\rho) g'(\Phi(\rho))\right] \rho^2 |\nabla u_2|^2 dA \\&+ O(\varepsilon)\|u\|_{L^2\left(\mathbb{S}^n\right)}^2+O(\varepsilon)\|\nabla u\|_{L^2\left(\mathbb{S}^n\right)}^2.
		\end{aligned}
	\end{equation}
	Notice that $n\geq3$, if $\varepsilon$ is sufficiently small and $\lambda$ is sufficiently large, the right-hand side of the latter inequality is nonnegative and therefore the statement follows.
	
	If equality holds in (\ref{weight gH volume ineq}), then from (\ref{defict g H int1}) we deduce $\nabla u_1 \equiv 0$ and $\nabla u_2 \equiv 0$. Then the unit normal $\nu$ is parallel to the radial direction $\partial_r$, and hence $M$ is a geodesic sphere centered at the origin.
	
\end{proof}

\begin{proof}[Proof of Theorem \ref{gH weight volume thm}]
	Assume that $\int_{\Omega}\phi' d\mathrm{vol} = \int_{\bar{B}_{\rho}}\phi' d\mathrm{vol}$, then
	plugging (\ref{u weight volum}) into (\ref{gPhiH int}) and using ${\phi'}^2 + K\phi^2 =1$ we have
		\begin{align*}
			&\int_{M} g(\Phi) H d\mu - \int_{\partial\bar{B}_{\rho}} g(\Phi) H d\mu \\=
			& - \int_{\mathbb{S}^n} n\left[(n-1)\phi^{n-3}(\rho) {\phi'}(\rho) g(\Phi(\rho))-\frac{1}{2}\phi^{n+1}(\rho) {\phi'}(\rho) g''(\Phi(\rho))\right.\\&\left.+\left(\frac{K}{2}\phi^{n+1}(\rho) {\phi'}^{-1}(\rho) - \frac{n-1}{2}\phi^{n-1}(\rho) {\phi'}(\rho)\right)\left(g'(\Phi(\rho))\phi'(\rho)+Kg(\Phi(\rho))\right)\right] \rho^2 u^2 dA \\&+ \int_{\mathbb{S}^n} \left[(n-1)\phi^{n-3}(\rho)\phi'(\rho) g(\Phi(\rho)) + \phi^{n-1}(\rho) g'(\Phi(\rho))\right] \rho^2 |\nabla u|^2 dA \\ & + \int_{\mathbb{S}^n} \frac{\phi^{n-2} g(\Phi) \rho^3}{D^2} \nabla^2 u \left[\nabla u, \nabla u\right] dA + O(\varepsilon)\|u\|_{L^2\left(\mathbb{S}^n\right)}^2+O(\varepsilon)\|\nabla u\|_{L^2\left(\mathbb{S}^n\right)}^2.
		\end{align*}
	Notice that $\phi'=1+\Phi$ in $\mathbb{H}^{n+1}$. Recall the assumption that $g$ is a non-decreasing convex $C^3$ positive function which satisfies $(1+s)g'(s)-g(s)\geq 0$, we have
	$$
	g'(\Phi(\rho))\phi'(\rho)-g(\Phi(\rho))\geq0.
	$$
	Then
	\begin{equation}\label{gH int weight volume}
		\begin{aligned}
			&\int_{M} g(\Phi) H d\mu - \int_{\partial\bar{B}_{\rho}} g(\Phi) H d\mu \\\geq
			& - \int_{\mathbb{S}^n} n(n-1)\phi^{n-3}(\rho) {\phi'}(\rho) g(\Phi(\rho)) \rho^2 u^2 dA \\&+ \int_{\mathbb{S}^n} (n-1)\phi^{n-3}(\rho)\phi'(\rho) g(\Phi(\rho)) \rho^2 |\nabla u|^2 dA \\ & + \int_{\mathbb{S}^n} \frac{\phi^{n-2} g(\Phi) \rho^3}{D^2} \nabla^2 u \left[\nabla u, \nabla u\right] dA + O(\varepsilon)\|u\|_{L^2\left(\mathbb{S}^n\right)}^2+O(\varepsilon)\|\nabla u\|_{L^2\left(\mathbb{S}^n\right)}^2.
		\end{aligned}
	\end{equation}
	Joining (\ref{gH int weight volume}) with the estimate (\ref{estimate}), we obtain
		\begin{align}
			&\int_{M} g(\Phi) H^+ d\mu - \int_{\partial\bar{B}_{\rho}} g(\Phi) H d\mu \notag\\ \geq
			& \int_{\mathbb{S}^n} (n-1)\phi^{n-3}(\rho)\phi'(\rho) g(\Phi(\rho)) \rho^2 |\nabla u_1|^2 dA \notag\\&- \int_{\mathbb{S}^n} n(n-1)\phi^{n-3}(\rho) {\phi'}(\rho) g(\Phi(\rho)) \rho^2 u_1^2 dA \label{defict g H int11}\\&+ \int_{\mathbb{S}^n} \frac{n-2}{2}\phi^{n-3}(\rho)\phi'(\rho) g(\Phi(\rho)) \rho^2 |\nabla u_2|^2 dA \notag \\ &  - \int_{\mathbb{S}^n} n(n-1)\phi^{n-3}(\rho) {\phi'}(\rho) g(\Phi(\rho)) \rho^2 u_2^2 dA \notag \\&+ O(\varepsilon)\|u\|_{L^2\left(\mathbb{S}^n\right)}^2+O(\varepsilon)\|\nabla u\|_{L^2\left(\mathbb{S}^n\right)}^2.\notag
		\end{align}
	Repeating the argument used to establish (\ref{weight volume poin esti}), we can also prove
	\begin{equation}\label{u1 weight volume poin esti1}
		\|\nabla u_1\|_{L^2\left(\mathbb{S}^n\right)}^2 \geq 2(n+1)\|u_1\|_{L^2\left(\mathbb{S}^n\right)}^2 + O(\varepsilon)\|u\|_{L^2\left(\mathbb{S}^n\right)}^2.
	\end{equation}
	Using the two Poincar\'{e}-type inequalities (\ref{u2 volume poin esti1}) and (\ref{u1 weight volume poin esti1}), the estimate (\ref{defict g H int11}) becomes 
	\begin{equation*}
		\begin{aligned}
			&\int_{M} g(\Phi) H^+ d\mu - \int_{\partial\bar{B}_{\rho}} g(\Phi) H d\mu \\\geq
			& \int_{\mathbb{S}^n} \frac{(n-1)(n+2)}{2(n+1)}\phi^{n-3}(\rho)\phi'(\rho) g(\Phi(\rho)) \rho^2 |\nabla u_1|^2 dA \\&+ \int_{\mathbb{S}^n} \left(\frac{n-2}{2}-\frac{n(n-1)}{\lambda}\right)\phi^{n-3}(\rho)\phi'(\rho) g(\Phi(\rho)) \rho^2 |\nabla u_2|^2 dA \\&+ O(\varepsilon)\|u\|_{L^2\left(\mathbb{S}^n\right)}^2+O(\varepsilon)\|\nabla u\|_{L^2\left(\mathbb{S}^n\right)}^2.
		\end{aligned}
	\end{equation*}
	
	Notice that $n\geq3$, if $\varepsilon$ is sufficiently small and $\lambda$ is sufficiently large, the right-hand side of the latter inequality is nonnegative and therefore the statement follows. We finish the proof by examining the equality  case similarly to Theorem \ref{gH volume thm}.
	
\end{proof}

\section{Stability of Weighted Quermassintegral inequalities in $\mathbb{H}^{n+1}$}\label{sec4}
In this section, we study the validity and the stability of geometric inequalities involving weighted curvature integrals and quermassintegrals for nearly spherical sets in $\mathbb{H}^{n+1}$. Based on the explicit formula for the weighted $k$th mean curvature integral established in Lemma \ref{lem expression g sigmak} and combining with the Poincar\'{e}-type estimate, we derive our main results.

\begin{proof}[Proof of Theorem \ref{gsigmak weight volume thm}]
	Assume that $\int_{\Omega}\phi' dv = \int_{\bar{B}_{\rho}}\phi' dv$, then
	plugging (\ref{u weight volum}) into (\ref{expression g sigmak}) and using ${\phi'}^2 + K\phi^2 =1$ we have
	\begin{align*}
		&\int_{M} g(\Phi)\sigma_{k}(\kappa) d\mu - \int_{\partial \bar{B}_{\rho}} g(\Phi)\sigma_{k}(\kappa) d\mu\\ 
			=& \int_{\mathbb{S}^n} \binom{n}{k} \left\{\left[-\frac{(n-k)(k+1)}{2} \phi^{n-k-2}(\rho) {\phi'}^{k}(\rho) + K\frac{k(k-2)}{2} \phi^{n-k}(\rho) {\phi'}^{k-2}(\rho)\right] g(\Phi(\rho))\right. \\&\left.- \left(k-\frac{1}{2}\right) \phi^{n-k}(\rho) {\phi'}^{k-1}(\rho) g'(\Phi(\rho)) + \frac{n}{2} \phi^{n-k}(\rho) {\phi'}^{k}(\rho) (\phi'(\rho)g'(\Phi(\rho))+Kg(\Phi(\rho)))\right. \\&\left. + \frac{1}{2} \phi^{n-k+2}(\rho) {\phi'}^{k}(\rho) g''(\Phi(\rho))\right\} \rho^2 u^2 dA \\
			&+ \int_{\mathbb{S}^n} \binom{n}{k} \left\{\left[\frac{(n-k)(k+1)}{2n} \phi^{n-k-2}(\rho) {\phi'}^{k}(\rho) - K \frac{k(k-1)}{2n} \phi^{n-k}(\rho) {\phi'}^{k-2}(\rho)\right] g(\Phi(\rho))\right. \\&\left.
			+ \frac{k}{n} \phi^{n-k}(\rho) {\phi'}^{k-1}(\rho) g'(\Phi(\rho))\right\} \rho^2 |\nabla u|^2 dA \\
			&+O(\varepsilon)\|u\|_{L^2\left(\mathbb{S}^n\right)}^2+O(\varepsilon)\|\nabla u\|_{L^2\left(\mathbb{S}^n\right)}^2.
	\end{align*}
	Recall the assumption that $g$ is a non-decreasing convex $C^3$ positive function which satisfies $(1+s)g'(s)-g(s)\geq 0$, we have
	\begin{equation}\label{gsigmak int weight volume}
		\begin{aligned}
			&\int_{M} g(\Phi)\sigma_{k}(\kappa) d\mu - \int_{\partial \bar{B}_{\rho}} g(\Phi)\sigma_{k}(\kappa) d\mu\\ 
			\geq& \int_{\mathbb{S}^n} \binom{n}{k} \left\{\left[-\frac{(n-k)(k+1)}{2} \phi^{n-k-2}(\rho) {\phi'}^{k}(\rho) + K\frac{k(k-2)}{2} \phi^{n-k}(\rho) {\phi'}^{k-2}(\rho)\right] g(\Phi(\rho))\right. \\&\left.- \left(k-\frac{1}{2}\right) \phi^{n-k}(\rho) {\phi'}^{k-1}(\rho) g'(\Phi(\rho))\right\} \rho^2 u^2 dA \\
			&+ \int_{\mathbb{S}^n} \binom{n}{k} \left\{\left[\frac{(n-k)(k+1)}{2n} \phi^{n-k-2}(\rho) {\phi'}^{k}(\rho) - K \frac{k(k-1)}{2n} \phi^{n-k}(\rho) {\phi'}^{k-2}(\rho)\right] g(\Phi(\rho))\right. \\&\left.
			+ \frac{k}{n} \phi^{n-k}(\rho) {\phi'}^{k-1}(\rho) g'(\Phi(\rho))\right\} \rho^2 |\nabla u|^2 dA \\
			&+O(\varepsilon)\|u\|_{L^2\left(\mathbb{S}^n\right)}^2+O(\varepsilon)\|\nabla u\|_{L^2\left(\mathbb{S}^n\right)}^2.
		\end{aligned}
	\end{equation}
	We notice that the coefficient of $\int_{\mathbb{S}^n}|\nabla u|^2 \mathrm{d} A$ in (\ref{gsigmak int weight volume}) is positive when $K=-1$. Indeed, by using the condition $(1+s)g'(s)-g(s)\geq 0$ and the fact $\phi'(\rho)=1+\Phi(\rho)$, we have
		\begin{align*}
			&\left[\frac{(n-k)(k+1)}{2n} \phi^{n-k-2}(\rho) {\phi'}^{k}(\rho) + \frac{k(k-1)}{2n} \phi^{n-k}(\rho) {\phi'}^{k-2}(\rho)\right] g(\Phi(\rho))
			\\&+ \frac{k}{n} \phi^{n-k}(\rho) {\phi'}^{k-1}(\rho) g'(\Phi(\rho))\\
			\geq&\left[\frac{(n-k)(k+1)}{2n} \phi^{n-k-2}(\rho) {\phi'}^{k}(\rho) + \frac{(k+1)k}{2n} \phi^{n-k}(\rho) {\phi'}^{k-2}(\rho)\right] g(\Phi(\rho))>0.
		\end{align*}
	Therefore, using the Poincar\'{e}-type estimate (\ref{weight volume poin esti}), the estimate (\ref{gsigmak int weight volume}) becomes
	\begin{equation}\label{defict g sigmk weight volume}
		\begin{aligned}
			&\int_{M} g(\Phi)\sigma_{k}(\kappa) d\mu - \int_{\partial \bar{B}_{\rho}} g(\Phi)\sigma_{k}(\kappa) d\mu\\ 
			\geq
			&\frac{1}{2} \binom{n}{k} \left\{\left[\frac{(n-k)(k+1)}{2n} \phi^{n-k-2}(\rho) {\phi'}^{k}(\rho)+ \frac{k(k-1)}{2n} \phi^{n-k}(\rho) {\phi'}^{k-2}(\rho)\right] g(\Phi(\rho))\right. \\&\left.
			+ \frac{k}{n} \phi^{n-k}(\rho) {\phi'}^{k-1}(\rho) g'(\Phi(\rho))\right\} \rho^2 \|\nabla u\|_{L^2\left(\mathbb{S}^n\right)}^2
			\geq 0.
		\end{aligned}
	\end{equation}
	We obtain the inequality (\ref{g sigmak weight volume ineq}) as desired.
	
	If equality holds in (\ref{g sigmak weight volume ineq}), then from (\ref{defict g sigmk weight volume}) we deduce $\nabla u \equiv 0$. Then the unit normal $\nu$ is parallel to the radial direction $\partial_r$, and hence $M$ is a geodesic sphere centered at the origin.
	
\end{proof}

\begin{proof}[Proof of Theorem \ref{gsgmk quermassintegral thm}]
	From $W_{-1}(\Omega)=W_{-1}(\bar{B}_{\rho})$, we have
	\begin{equation}\label{volume condition11}
		\int_{\mathbb{S}^n} u dA=-\frac{n}{2} \frac{\phi'(\rho)}{\phi(\rho)} \rho \int_{\mathbb{S}^n} u^2 dA +O(\varepsilon)\|u\|_{L^2\left(\mathbb{S}^n\right)}^2,
	\end{equation}
	and from $W_j(\Omega)=W_j(\bar{B}_{\rho})(0 \leq j<k)$, we have
	\begin{equation}\label{general condition11}
		\begin{aligned}
			\int_{\mathbb{S}^n} u dA= & -\int_{\mathbb{S}^n} \left[\frac{n-j-1}{2} \frac{\phi'(\rho)}{\phi(\rho)} - K\frac{j+1}{2} \frac{\phi(\rho)}{\phi'(\rho)}\right] \rho u^2 dA \\& - \int_{\mathbb{S}^n} \frac{j+1}{2 n} \frac{1}{\phi(\rho)\phi'(\rho)} \rho |\nabla u|^2 dA+O(\varepsilon)\|u\|_{L^2\left(\mathbb{S}^n\right)}^2 \\
			& +O(\varepsilon)\|\nabla u\|_{L^2\left(\mathbb{S}^n\right)}^2.
		\end{aligned}
	\end{equation}
	These are proved in \cite[Lemma 4.2]{ZZ23}. Then by substituting (\ref{volume condition11}) in (\ref{expression g sigmak}) when $j=-1$ and substituting (\ref{general condition11}) in (\ref{expression g sigmak}) when $0 \leq j<k$, using ${\phi'}^2 + K\phi^2 =1$ we have
		\begin{align*}
			&\int_{M} g(\Phi)\sigma_{k}(\kappa) d\mu - \int_{\partial \bar{B}_{\rho}} g(\Phi)\sigma_{k}(\kappa) d\mu\\ 
			=&\int_{\mathbb{S}^n} \binom{n}{k} \left\{\left[\frac{(n-k)(j-k)}{2} \phi^{n-k-2}(\rho) {\phi'}^{k}(\rho) + K\frac{k(k-j-2)}{2} \phi^{n-k}(\rho) {\phi'}^{k-2}(\rho)\right] g(\Phi(\rho))\right. \\&\left.+ \frac{n+1}{2} \phi^{n-k}(\rho) {\phi'}^{k+1}(\rho) g'(\Phi(\rho)) + \left(\frac{j+1}{2} - k\right) \phi^{n-k}(\rho) {\phi'}^{k-1}(\rho) g'(\Phi(\rho))\right. \\&\left. + \frac{1}{2} \phi^{n-k+2}(\rho) {\phi'}^{k}(\rho) g''(\Phi(\rho))\right\} \rho^2 u^2 dA \\
			&+ \int_{\mathbb{S}^n} \binom{n}{k} \left\{\left[\frac{(n-k)(k-j)}{2n} \phi^{n-k-2}(\rho) {\phi'}^{k}(\rho)\right.\right. \\&\left.\left.- K \frac{k(k-j-2)}{2n} \phi^{n-k}(\rho) {\phi'}^{k-2}(\rho)\right] g(\Phi(\rho))\right. \\&\left.
			+ \frac{2k-j-1}{2n} \phi^{n-k}(\rho) {\phi'}^{k-1}(\rho) g'(\Phi(\rho))\right\} \rho^2 |\nabla u|^2 dA \\
			&+O(\varepsilon)\|u\|_{L^2\left(\mathbb{S}^n\right)}^2+O(\varepsilon)\|\nabla u\|_{L^2\left(\mathbb{S}^n\right)}^2.
		\end{align*}
	Recall the assumption that $g$ is a non-decreasing convex $C^3$ positive function, we have 
		\begin{align}
			&\int_{M} g(\Phi)\sigma_{k}(\kappa) d\mu - \int_{\partial \bar{B}_{\rho}} g(\Phi)\sigma_{k}(\kappa) d\mu \notag\\ 
			\geq&\int_{\mathbb{S}^n} \binom{n}{k} \left\{\left[\frac{(n-k)(j-k)}{2} \phi^{n-k-2}(\rho) {\phi'}^{k}(\rho)\right.\right. \notag\\&\left.\left.+ K\frac{k(k-j-2)}{2} \phi^{n-k}(\rho) {\phi'}^{k-2}(\rho)\right] g(\Phi(\rho))\right. \notag\\&\left. + \frac{j+1-2k}{2} \phi^{n-k}(\rho) {\phi'}^{k-1}(\rho) g'(\Phi(\rho))\right\} \rho^2 u^2 dA \notag\\
			&+ \int_{\mathbb{S}^n} \binom{n}{k} \left\{\left[\frac{(n-k)(k-j)}{2n} \phi^{n-k-2}(\rho) {\phi'}^{k}(\rho)\right.\right. \label{defict g sigmak}\\&\left.\left.- K \frac{k(k-j-2)}{2n} \phi^{n-k}(\rho) {\phi'}^{k-2}(\rho)\right] g(\Phi(\rho))\right. \notag\\&\left.
			+ \frac{2k-j-1}{2n} \phi^{n-k}(\rho) {\phi'}^{k-1}(\rho) g'(\Phi(\rho))\right\} \rho^2 |\nabla u|^2 dA \notag\\
			&+O(\varepsilon)\|u\|_{L^2\left(\mathbb{S}^n\right)}^2+O(\varepsilon)\|\nabla u\|_{L^2\left(\mathbb{S}^n\right)}^2.\notag
		\end{align}
	We notice that the coefficient of $\int_{\mathbb{S}^n}|\nabla u|^2 \mathrm{d} A$ in (\ref{defict g sigmak}) is positive when $K=-1$. Indeed, by using the condition $(1+s)g'(s)-g(s)\geq 0$ and the fact $\phi'(\rho)=1+\Phi(\rho)$, we have
		\begin{align*}
			&\left[\frac{(n-k)(k-j)}{2n} \phi^{n-k-2}(\rho) {\phi'}^{k}(\rho) + \frac{k(k-j-2)}{2n} \phi^{n-k}(\rho) {\phi'}^{k-2}(\rho)\right] g(\Phi(\rho))
			\\&+ \frac{2k-j-1}{2n} \phi^{n-k}(\rho) {\phi'}^{k-1}(\rho) g'(\Phi(\rho))\\
			\geq&\left[\frac{(n-k)(k-j)}{2n} \phi^{n-k-2}(\rho) {\phi'}^{k}(\rho) + \frac{(k+1)(k-j-1)}{2n} \phi^{n-k}(\rho) {\phi'}^{k-2}(\rho)\right] g(\Phi(\rho))>0.
		\end{align*}
    Therefore, using the Poincar\'{e}-type estimate in \cite[Lemma 4.2]{ZZ23}, the estimate (\ref{defict g sigmak}) becomes
		\begin{align*}
			&\int_{M} g(\Phi)\sigma_{k}(\kappa) d\mu - \int_{\partial \bar{B}_{\rho}} g(\Phi)\sigma_{k}(\kappa) d\mu\\ 
			\geq
			&\frac{1}{2} \binom{n}{k} \left\{\left[\frac{(n-k)(k-j)}{2n} \phi^{n-k-2}(\rho) {\phi'}^{k}(\rho)+ \frac{k(k-j-2)}{2n} \phi^{n-k}(\rho) {\phi'}^{k-2}(\rho)\right] g(\Phi(\rho))\right. \\&\left.
			+ \frac{2k-j-1}{2n} \phi^{n-k}(\rho) {\phi'}^{k-1}(\rho) g'(\Phi(\rho))\right\} \rho^2 \|\nabla u\|_{L^2\left(\mathbb{S}^n\right)}^2\\
			\geq& \binom{n}{k}\frac{(n-k)(k-j)}{4n} \phi^{n-k-2}(\rho) {\phi'}^{k}(\rho)g(\Phi(\rho)) \rho^2 \|\nabla u\|_{L^2\left(\mathbb{S}^n\right)}^2.
		\end{align*}
	Combining the estimate proved in \cite[Lemma 4.3]{ZZ23}:
	$$
	{\alpha}^2(\Omega) \leq \frac{1}{n^2} \operatorname{Area}\left(\mathbb{S}^n\right) \phi^{2n}(\rho) \rho^2 \|\nabla u\|_{L^2\left(\mathbb{S}^n\right)}^2+O(\varepsilon)\|\nabla u\|_{L^2\left(\mathbb{S}^n\right)}^2,
	$$
	we obtain the inequality 
	\begin{equation*}
		\begin{aligned}
			&\int_{M} g(\Phi)\sigma_{k}(\kappa) d\mu - \int_{\partial \bar{B}_{\rho}} g(\Phi)\sigma_{k}(\kappa) d\mu\\ 
			\geq& \left(\binom{n}{k}\frac{n(n-k)(k-j)}{4} \frac{ {\phi'}^{k}(\rho)g(\Phi(\rho))}{\operatorname{Area}\left(\mathbb{S}^n\right)\phi^{n+k+2}(\rho)}+O(\varepsilon)\right) \alpha^2 (\Omega)
		\end{aligned}
	\end{equation*}
	as desired.
	
\end{proof}

\section{Stability of Weighted Quermassintegral inequalities in $\mathbb{R}^{n+1}$}\label{sec5}
Finally, we establish quantitative weighted quermassintegral inequalities in $\mathbb{R}^{n+1}$. The proof of Theorem \ref{thm gsk rn} follows a similar method to that of Theorem \ref{gsgmk quermassintegral thm} in hyperbolic space, though slight differences arise in the geometric analysis due to the vanishing curvature of the Euclidean setting. 

\begin{proof}[Proof of Theorem \ref{thm gsk rn}]
	Taking $K=0$, $\phi(r)=r$ in (\ref{expression g sigmak}), we get that for any nearly spherical set $M=\{(\rho(1+u(x)), x): x \in \mathbb{S}^n\}$ in $\mathbb{R}^{n+1}$,
	\begin{align}
		&\int_{M} g(\Phi)\sigma_{k}(\kappa) d\mu \notag\\ 
		=&\int_{\mathbb{S}^n} \binom{n}{k} \rho^{n-k} g(\Phi(\rho)) dA \notag\\
		&+\int_{\mathbb{S}^n} \binom{n}{k} \left\{(n-k)\rho^{n-k-1} g(\Phi(\rho))+ \rho^{n-k+1} g'(\Phi(\rho))\right\} \rho u dA \notag\\
		&+ \int_{\mathbb{S}^n} \binom{n}{k} \left\{\frac{(n-k)(n-k-1)}{2} \rho^{n-k-2} g(\Phi(\rho))\right. \label{expression g sigmak rn}\\&\left.+ \left(n-k+\frac{1}{2}\right) \rho^{n-k} g'(\Phi(\rho)) + \frac{1}{2} \rho^{n-k+2} g''(\Phi(\rho))\right\} \rho^2 u^2 dA \notag\\
		&+ \int_{\mathbb{S}^n} \binom{n}{k} \left\{\frac{(n-k)(k+1)}{2n} \rho^{n-k-2} g(\Phi(\rho))
		+ \frac{k}{n} \rho^{n-k} g'(\Phi(\rho))\right\} \rho^2 |\nabla u|^2 dA \notag\\
		&+O(\varepsilon)\|u\|_{L^2\left(\mathbb{S}^n\right)}^2+O(\varepsilon)\|\nabla u\|_{L^2\left(\mathbb{S}^n\right)}^2.\notag
	\end{align}
	From $W_{-1}(\Omega)=W_{-1}(\bar{B}_{\rho})$, we have
	\begin{equation}\label{volume condition1}
		\int_{\mathbb{S}^n} u d A=-\int_{\mathbb{S}^n} \frac{n}{2} u^2 d A+O(\varepsilon)\|u\|_{L^2\left(\mathbb{S}^n\right)}^2,
	\end{equation}
	and from $W_j(\Omega)=W_j(\bar{B}_{\rho})(0 \leq j<k)$, we have
	\begin{equation}\label{general condition1}
		\begin{aligned}
			\int_{\mathbb{S}^n} u d A=- & \int_{\mathbb{S}^n} \frac{n-j-1}{2} u^2 dA-\int_{\mathbb{S}^n} \frac{j+1}{2 n}|\nabla u|^2 d A+O(\varepsilon)\|u\|_{L^2\left(\mathbb{S}^n\right)}^2 \\
			& +O(\varepsilon)\|\nabla u\|_{L^2\left(\mathbb{S}^n\right)}^2.
		\end{aligned}
	\end{equation}
	These are proved in Proposition 4.3 and Lemma 5.2 of \cite{VW24}. Then by substituting (\ref{volume condition1}) in (\ref{expression g sigmak rn}) when $j=-1$ and substituting (\ref{general condition1}) in (\ref{expression g sigmak rn}) when $0 \leq j<k$,
		\begin{align*}
			&\int_M g(\Phi) \sigma_k(\kappa) d \mu -\int_{\partial\bar{B}_{\rho}} g(\Phi)\sigma_{k}(\kappa)d\mu\\
			=&\int_{\mathbb{S}^n}\binom{n}{k}\left\{\frac{(n-k)(j-k)}{2}\rho^{n-k+2}g(\Phi(\rho))+\frac{n-2k+j+2}{2}\rho^{n-k}g'(\Phi(\rho))\right. \\ &\left.+\frac{1}{2}\rho^{n-k+2}g''(\Phi(\rho))\right\}\rho^{2}u^{2}\\
			&+\int_{\mathbb{S}^n}\binom{n}{k}\left\{\frac{(n-k)(k-j)}{2n}\rho^{n-k-2}g(\Phi(\rho))+\frac{2k-j-1}{2n}\rho^{n-k}g'(\Phi(\rho))\right\}\rho^{2} |\nabla u|^{2}.
		\end{align*}
	Recall the assumption that $g$ is a non-decreasing convex $C^3$ positive function, we have 
	\begin{equation}\label{coefficient1}
		\begin{aligned}
			&\int_M g(\Phi) \sigma_k(\kappa) d \mu -\int_{\partial\bar{B}_{\rho}} g(\Phi)\sigma_{k}(\kappa)d\mu\\
			\geq&\int_{\mathbb{S}^n}\binom{n}{k}\left\{\frac{(n-k)(j-k)}{2}\rho^{n-k+2}g(\Phi(\rho))+\frac{n-2k+j+2}{2}\rho^{n-k}g'(\Phi(\rho))\right\}\rho^{2}u^{2}\\
			&+\int_{\mathbb{S}^n}\binom{n}{k}\left\{\frac{(n-k)(k-j)}{2n}\rho^{n-k-2}g(\Phi(\rho))+\frac{2k-j-1}{2n}\rho^{n-k}g'(\Phi(\rho))\right\}\rho^{2} |\nabla u|^{2}.
		\end{aligned}
	\end{equation}
	We remark that the coefficient of $\int_{\mathbb{S}^n}|\nabla u|^2 \mathrm{d} A$ in (\ref{coefficient1}) is positive. Therefore, using the Poincar\'{e}-type estimate in \cite[Lemma 5.2]{VW24}, 
	$$
	\|\nabla u\|_{L^2\left(\mathbb{S}^n\right)}^2 \geqslant 2(n+1)\|u\|_{L^2\left(\mathbb{S}^n\right)}^2+O(\varepsilon)\|u\|_{L^2\left(\mathbb{S}^n\right)}^2
	$$
	(when $j=-1$ ) and
	$$
	\|\nabla u\|_{L^2\left(\mathbb{S}^n\right)}^2 \geqslant 2(n+1)\|u\|_{L^2\left(\mathbb{S}^n\right)}^2+O(\varepsilon)\|u\|_{L^2\left(\mathbb{S}^n\right)}^2+O(\varepsilon)\|\nabla u\|_{L^2\left(\mathbb{S}^n\right)}^2
	$$
	(when $0 \leqslant j<k$ ), the estimate (\ref{coefficient1}) becomes
	\begin{align*}
		&\int_{M} g(\Phi)\sigma_{k}(\kappa) d\mu - \int_{\partial\bar{B}_{\rho}} g(\Phi)\sigma_{k}(\kappa)d\mu\\ 
		\geq
		&\frac{1}{2} \binom{n}{k} \left\{\frac{(n-k)(k-j)}{2n} \rho^{n-k-2} g(\Phi(\rho))
		+ \frac{2k-j-1}{2n} \rho^{n-k} g'(\Phi(\rho))\right\} \rho^2 \|\nabla u\|_{L^2\left(\mathbb{S}^n\right)}^2\\
		=& \binom{n}{k}\frac{(n-k)(k-j)\rho^{n-k} g(\Phi(\rho))+(2k-j-1)\rho^{n-k+2} g'(\Phi(\rho))}{4n}\|\nabla u\|_{L^2\left(\mathbb{S}^n\right)}^2.
	\end{align*}
	Combining the estimate proved in \cite[Lemma 5.3]{VW24}:
	$$
	{\alpha}^2(\Omega) \leqslant \frac{1}{n^2}\operatorname{Area}\left(\mathbb{S}^n\right) \rho^{2n+2} \|\nabla u\|_{L^2\left(\mathbb{S}^n\right)}^2+O(\varepsilon)\|\nabla u\|_{L^2\left(\mathbb{S}^n\right)}^2,
	$$
	we obtain the inequality 
	\begin{equation*}
		\begin{aligned}
			&\int_{M} g(\Phi)\sigma_{k}(\kappa) d\mu - \int_{\partial\bar{B}_{\rho}} g(\Phi)\sigma_{k}(\kappa)d\mu\\ 
			\geq& \left[\binom{n}{k}\frac{n((n-k)(k-j)g(\Phi(\rho))+(2k-j-1)\rho^{2}g'(\Phi(\rho)))}{4\operatorname{Area}(\mathbb{S}^n) \rho^{n+k+2}}+O(\varepsilon)\right] \alpha^2 (\Omega)
		\end{aligned}
	\end{equation*}
	as desired.
	
\end{proof}

\end{document}